\documentclass[11pt]{amsart}

\usepackage{amssymb, amsthm, amsmath, amsfonts, amsxtra, mathrsfs, mathtools, setspace}
\usepackage{color}
\usepackage{import}
\usepackage{thmtools}
\usepackage{thm-restate}
\usepackage[colorlinks=false, allcolors=blue]{hyperref}
\usepackage[capitalize]{cleveref}
\usepackage{chngcntr}
\usepackage[intoc]{nomencl}
\makenomenclature
\usepackage{stmaryrd}
\usepackage{svg}
\usepackage{tikz-cd}
\usepackage{tikz}
\usepackage{tikzsymbols}
\usetikzlibrary{decorations.pathreplacing,angles,quotes}
\usepackage{enumitem}
\usetikzlibrary{matrix}
\usepackage[ruled,vlined]{algorithm2e}
\usepackage[mathscr]{euscript}
\usepackage[normalem]{ulem}
\usepackage{comment}
\usepackage{multicol}
\usepackage{caption, subcaption}

\input xy
\xyoption{all}

\usepackage[letterpaper, top=3cm, bottom=3cm,left=2.7cm, right=2.7cm, heightrounded,bindingoffset=0mm]{geometry}

\newtheorem{theorem}{Theorem}[section]

\newtheorem{corollary}[theorem]{Corollary}
\newtheorem{lemma}[theorem]{Lemma}
\newtheorem{proposition}[theorem]{Proposition}
\theoremstyle{definition}

\newtheorem{remark}[theorem]{Remark}
\newtheorem{question}[theorem]{Question}

\def\Z{\mathbb{Z}}
\def\Q{\mathbb{Q}}
\def\R{\mathbb{R}}

\def\P{\mathbb{P}}

\def\N{\mathbb{N}}
\def\Z{\mathbb{Z}}

\def\calA{\mathcal{A}}
\def\calB{\mathcal{B}}

\def\calF{\mathcal{F}}

\def\calK{\mathcal{K}}
\def\calL{\mathcal{L}}
\def\calM{\mathcal{M}}

\def\calO{\mathcal{O}}

\newcommand{\IP}{I\! P}

\newcommand{\excise}[1]{}
\newcommand{\abs}[1]{\left\lvert#1\right\rvert}
\renewcommand{\and}{\qquad\text{and}\qquad}

\newcommand{\ch}{\textup{ch}}

\newcommand{\gr}{\operatorname{gr}}

\newcommand{\Hom}{\operatorname{Hom}}

\newcommand{\la}{\lambda}

\newcommand{\rk}{\operatorname{rk}}

\newcommand{\Snap}{\textup{Snap}}

\newcommand{\cA}{\mathcal{A}}
\newcommand{\WA}{W_{\!\cA}}
\newcommand{\Waug}{W^{\aug}_{\!\cA}}
\newcommand{\aug}{{\operatorname{aug}}}
\renewcommand{\and}{\qquad\text{and}\qquad}
\newcommand{\cJ}{\mathcal{J}}
\newcommand{\cI}{\mathcal{I}}
\newcommand{\cO}{\mathcal{O}}
\newcommand{\cL}{\mathcal{L}}
\newcommand{\cF}{\mathcal{F}}

\newcommand{\PhFY}{\Phi_{\operatorname{FY}}}
\newcommand{\Phsim}{\Phi_\nabla}
\newcommand{\PsFY}{\Psi_{\operatorname{FY}}}
\newcommand{\Pssim}{\Psi_\nabla}

\newcommand{\cU}{\mathcal{U}}
\newcommand{\FY}{\operatorname{FY}}
\newcommand{\Flag}{\operatorname{Flag}}

\title{$K$-rings of wonderful varieties and matroids}

\author[M. Larson]{Matt Larson}\address{Department of Mathematics, Stanford University, Stanford, CA 94305}
\email{\url{mwlarson@stanford.edu}}

\author[S. Li]{Shiyue Li}\address{Department of Mathematics, Brown University, Providence, RI 02906}
\email{\url{shiyue_li@brown.edu}}

\author[S. Payne]{Sam Payne}\address{Department of Mathematics, University of Texas, Austin, TX 78712}\email{\url{sampayne@utexas.edu}}

\author[N. Proudfoot]{Nicholas Proudfoot}\address{Department of Mathematics, University of Oregon, Eugene, OR 97403}
\email{\url{njp@uoregon.edu}}
\date{\today}

\begin{document}
\spacing{1.2}

\begin{abstract}
We study the $K$-ring of the wonderful variety of a hyperplane arrangement and give a combinatorial presentation that depends only on the underlying matroid. We use this combinatorial presentation to define the $K$-ring of an arbitrary loopless matroid. We construct an exceptional isomorphism, with integer coefficients, to the Chow ring of the matroid that satisfies a Hirzebruch--Riemann--Roch-type formula, generalizing a recent construction of Berget, Eur, Spink, and Tseng for the permutohedral variety (the wonderful variety of a Boolean arrangement).  As an application, we give combinatorial formulas for Euler characteristics of arbitrary line bundles on wonderful varieties. We give analogous constructions and results for augmented wonderful varieties, and for Deligne--Mumford--Knudsen moduli spaces of stable rational curves with marked points.
\end{abstract}

\renewcommand{\baselinestretch}{.95}\normalsize
\maketitle
\setcounter{tocdepth}{1}
\tableofcontents
\renewcommand{\baselinestretch}{1.2}\normalsize

\section{Introduction}
Let $L$ be a finite dimensional vector space over a field $\mathbb{F}$, and let $\cA = \{H_e\mid e\in E\}$ be a finite 
multiset of hyperplanes in $L$ intersecting only at the origin.  
The wonderful variety $\WA$ is a smooth compactification of $\mathbb{P}(L)\setminus\bigcup_{H\in\cA}\mathbb{P}(H)$,
originally studied by De Concini and Procesi \cite{dCP95}.  The augmented wonderful variety $\WA^{\aug}$ is a smooth 
compactification of $L$, introduced in \cite{BHMPW20a}, that contains $\WA$ as a divisor.
The Chow rings of these spaces have been extensively studied and have combinatorial presentations that depend only on the underlying matroid. As a result, such rings are naturally defined for arbitrary, not necessarily realizable, matroids.
Presentations for the Chow ring of $\WA$
appear in \cite{dCP95,FY,BES}, and this ring is used to prove log concavity
of the characteristic polynomial of $\cA$ (and more generally of any matroid) in \cite{AHK}.
A presentation of the Chow ring of $\Waug$ appears in \cite{BHMPW20a}, and this ring is used to prove the top-heavy conjecture and the nonnegativity of Kazhdan--Lusztig polynomials of matroids  \cite{BHMPW20b}.

Our goal is to study the Grothendieck rings of vector bundles $K(\WA)$ and $K(\Waug)$.  We first give presentations for the $K$-rings using generators analogous to the generators given by Feichtner--Yuzvinsky for the
Chow ring of $W_{\mathcal{A}}$.  
Despite the fact that the (non-homogeneous) relations among the Feichtner--Yuzvinsky generators in the $K$-ring are different from the (homogeneous) relations
among the Feichtner--Yuzvinsky generators in the Chow ring, we
construct 
an integral isomorphism from the $K$-ring to the Chow ring which satisfies a Hirzebruch--Riemann--Roch-type formula, generalizing recent results for Boolean
arrangements \cite{BEST,EurHuhLarson}.  The $K$-rings admit additional structures,
such as Adams operations, an Euler characteristic map, and Serre duality, which leads us to new results and new questions.
Most of our results can be extended to the moduli space $\overline{\mathcal{M}}_{0,n}$ of stable rational curves with $n$ 
marked points, which is closely related to the wonderful variety for the braid arrangement $\mathcal{B}_{n-1}$. 

\subsection{Definitions of the varieties}
For any $S\subset E$, let 
$$L_S \coloneqq \bigcap_{e\in S} H_e \and L^S \coloneqq L/L_S.$$
The dimension of $L^S$ is called the {\bf rank} of $S$, and the dimension of $L_S$ is called the {\bf corank}.
A subset $F\subset E$ is called a {\bf flat} if it is maximal within its rank, or, equivalently, if $L_F \subset H_e$ implies that $e\in F$.

The {\bf wonderful variety} $\WA$ is defined as the closure of the image of the rational map
$$\mathbb{P}(L) \dasharrow \prod_{F \not= \emptyset} \mathbb{P}(L^F),$$
and the {\bf augmented wonderful variety} $\Waug$ is defined as the closure of the image of the rational map
$$\mathbb{P}(L\oplus \mathbb{F}) \dasharrow \prod_{F \not= \emptyset} \mathbb{P}(L^F\oplus\mathbb{F}),$$
where both products are over the set of nonempty flats.
For any nonempty flat $F$, let $$\pi_F \colon \WA\to \mathbb{P}(L^F)\and\pi^\aug_F \colon \Waug\to \mathbb{P}(L^F\oplus\mathbb{F})$$
be the natural projections.  We will write $$\cL_F \coloneqq \pi_F^*\cO_{\mathbb{P}(L^F)}(1)\and \cL_F^\aug \coloneqq (\pi_F^\aug)^*\cO_{\mathbb{P}(L^F\oplus\,\mathbb{F})}(1).$$

The map $\pi_{E} \colon \WA\to \mathbb{P}(L)$ is an iterated blow-up. It is obtained by blowing up first the points $\mathbb{P}(L_F)$ for all corank 1 flats $F$, then the strict transforms of the lines $\mathbb{P}(L_F)$ for all corank 2 flats $F$, and so on. For any nonempty proper flat $F$, let $D_F\subset\WA$ be the closure of the preimage under the map $\pi_E$ of the locus 
$$\mathbb{P}(L_F) \setminus \bigcup_{F\subsetneq G \not= E} \mathbb{P}(L_G).$$
Similarly, the map $\pi^\aug_{E} \colon \Waug\to \mathbb{P}(L\oplus\mathbb{F})$
is an iterated blow-up, obtained by blowing up first the points $\mathbb{P}(L_F\oplus\{0\})$ for all corank 1 flats $F$, then the strict transforms of the lines $\mathbb{P}(L_F\oplus\{0\})$ for all corank 2
flats $F$, and so on.  For any proper flat $F$, let $D_F^\aug\subset\Waug$ be the 
closure of the preimage under the map $\pi_E^\aug$ of the locus 
$$\mathbb{P}(L_F\oplus \{0\}) \setminus \bigcup_{F\subsetneq G \not= E} \mathbb{P}(L_G\oplus\{0\}).$$
In addition, for any $e\in E$, let $D_e^\aug\subset \Waug$ be the strict transform of $\mathbb{P}(H_e\oplus\mathbb{F})$.
We have a canonical isomorphism $\WA\cong D_{\varnothing}^\aug\subset \Waug$,
which induces identifications $\cL_F = \cL_F^\aug|_{\WA}$ and $D_F = D_F^\aug\cap \WA$ for all nonempty proper $F$.

\subsection{Feichtner--Yuzvinsky presentations}\label{sec:intro-toric}
Let $A(\WA)$ denote the Chow ring of cycles modulo rational equivalence on $\WA$, see \cite{Fulton}. Similarly, let $A(\Waug)$ denote the Chow ring of $\Waug$. Let
$T_\cA \coloneqq \Z[x_F\mid \text{$F$ a flat}] \otimes \Z[y_e\mid e\in E].$
Consider the map $$\PhFY^\aug\colon T_\cA \to A(\Waug)$$ that sends $x_F$ to $[D_F^\aug]$ for all proper $F$, $x_E$ to $-c_1(\cL_E^\aug)$, and 
$y_e$ to $[D_e^\aug]$ for all $e\in E$.  Composing with the pullback along the inclusion $\WA\subset\Waug$, we obtain a map 
$$\PhFY\colon T_\cA\to A(\WA)$$ that
sends $x_F$ to $[D_F]$ for all nonempty proper $F$ and $y_e$ to zero. The maps $\PhFY$ and $\PhFY^\aug$ are surjective, and their kernels were explicitly described in \cite{FY,BHMPW20a} in terms of the flats of $\mathcal{A}$ (Theorem~\ref{FYChow}).

Next, 
consider the map $$\PsFY^\aug\colon T_\cA \to K(\Waug)$$ that sends $x_F$ to $[\cO_{D_F^\aug}]$ for all proper $F$, $x_E$ to 
$1 - [\cL_E^\aug]$, 
and $y_e$ to $[\cO_{D_e^\aug}]$
for all $e\in E$.  Composing with the pullback along the inclusion $\WA\subset\Waug$, we obtain a map $$\PsFY\colon T_\cA\to K(\WA)$$ that
sends $x_F$ to $[\cO_{D_F}]$ for all nonempty proper $F$ and $y_e$ to zero. We show that the maps $\PsFY$ and $\PsFY^{\aug}$ are surjective and give an explicit description of the kernels in terms of the flats of $\mathcal{A}$ (Theorem~\ref{FYK}). The relations generating the kernels of $\PsFY$ and $\PsFY^{\aug}$ are ``inhomogeneous versions'' of the relations generating the kernels of $\PhFY$ and $\PhFY^{\aug}$.

\subsection{Exceptional isomorphisms and simplicial presentations}\label{sec:intro-simplicial}
For any nonempty flat $F$, we define the following Chow and $K$-classes:
\begin{equation*}
\begin{aligned}
h_F &\coloneqq c_1(\cL_F)\in A(\WA) & \quad \quad 
h_F^\aug &\coloneqq c_1(\cL_F^\aug)\in A(\Waug)\\
\eta_F &\coloneqq 1 - [\cL_F^{-1}] \in K(\WA) & 
\eta_F^\aug &\coloneqq 1 - [(\cL_F^\aug)^{-1}]\in K(\Waug).
\end{aligned}
\end{equation*}

\begin{remark}
The Chow classes $\{h_F\}$ and $\{h_F^\aug\}$ may be interpreted in terms of divisors on the permutohedral or stellahedral  variety coming from simplices and are therefore known as {\bf simplicial generators}.  In the non-augmented setting, these classes were studied in \cite{YuzvinskySimplicial,BES}, and the definitions and basic properties immediately generalize to the augmented setting. See \cite[Section 3.2]{BES} for more details.  We will similarly refer to $\{\eta_F\}$ and $\{\eta_F^\aug\}$
as the simplicial generators of the $K$-ring.  The statement that they generate their respective $K$-rings is true but not obvious.
\end{remark}

Our next result
says that there are isomorphisms from the $K$-ring to the Chow ring of the varieties $\WA$ and $\Waug$ that satisfy  analogues of the Hirzebruch--Riemann--Roch formula, 
in which $\frac{1}{1-h_E}$ and $\frac{1}{1-h_E^\aug}$ play the roles of the respective Todd classes.
The special case in which $\cA$ is the Boolean arrangement (i.e., the number of hyperplanes is equal to the dimension of $L$)
appeared in \cite[Theorem D]{BEST} in the non-augmented setting and in \cite[Theorem 1.8]{EurHuhLarson} in the augmented setting.  Our techniques give new proofs of these results.

\begin{theorem}\label{exceptional}
There are isomorphisms
$\zeta_\cA\colon K(\WA)\to A(\WA)$ and $\zeta_\cA^\aug\colon K(\Waug)\to A(\Waug)$
characterized by the property that $\zeta_\cA(\eta_F) = h_F$ and 
$\zeta_\cA^\aug(\eta^\aug_F) = h^\aug_F.$
For any classes $\xi\in K(\WA)$ and $\xi^\aug\in K(\Waug)$, we have
$$\chi(\WA, \xi) = \deg_{\WA}\left(\frac{\zeta_\cA(\xi)}{1-h_E}\right) \and \chi(\Waug, \xi^\aug) = \deg_{\Waug}\left(\frac{\zeta_\cA^\aug(\xi^\aug)}{1-h_E^\aug}\right).$$
\end{theorem}

\begin{remark}
The isomorphisms of Theorem \ref{exceptional} are not related to the Chern character homomorphisms; 
in particular, they do not coincide with the respective Chern characters after tensoring with the rational numbers.
If we replaced $\zeta_\cA$ (respectively $\zeta_\cA^\aug$) with the Chern character, 
a similar formula would hold with $\frac{1}{1-h_E}$ (respectively $\frac{1}{1-h_E^\aug}$)
replaced by the Todd class.
\end{remark}

Let $S_\cA \coloneqq \Z[u_F \mid \text{$F$ a nonempty flat}\},$ and
consider the map $$\Phsim^\aug\colon S_\cA \to A(\Waug)$$ that sends $u_F$ to $h^\aug_F$ for all $F$.  Composing with the pullback along the inclusion $\WA\subset\Waug$, we obtain a map $$\Phsim\colon S_\cA\to A(\WA)$$ that
sends $u_F$ to $h_F$ for all $F$. The maps $\Phsim$ and $\Phsim^{\aug}$ are surjective, and the ideal of relations can be computed explicitly in terms of the arrangement (Theorem~\ref{Amatroid}). This is called the \textbf{simplicial presentation} of $A(M)$ and $A^{\aug}(M)$.

For the $K$-theoretic analogue of the simplicial generators, consider the map 
$$\Pssim^\aug\colon S_\cA \to K(\Waug)$$ that sends $u_F$ to $\eta^\aug_F$ for all $F$.  
Composing with the pullback along the inclusion $\WA\subset\Waug$, we obtain a map $$\Pssim\colon S_\cA\to K(\WA)$$ that
sends $u_F$ to $\eta_F$ for all $F$.
Theorems \ref{exceptional} immediately implies that the kernel of $\Pssim$ is the same as the kernel of $\Phsim$, and that the kernel of $\Pssim^{\aug}$ is the same as the kernel of $\Phsim^{\aug}$. This kernel is computed explicitly in Theorem~\ref{Amatroid}, giving what we refer to as the \textbf{simplicial presentation} of $K(M)$ and $K^{\aug}(M)$.

\subsection{The $K$-ring of $\overline{\mathcal{M}}_{0,n}$}\label{ssec:m0n}
\label{sec:intro-m0n}
Our results generalize to $\overline{\calM}_{0, n}$, the Deligne--Mumford--Knudsen compactification of the moduli space of stable rational curves with $n$ marked points, for $n \ge 3$. For $i \in \{1, \dotsc, n\}$, there is a line bundle $\mathbb{L}_i$ on $\overline{\mathcal{M}}_{0,n}$ whose fiber over a point is the cotangent space of the $i$th marked point in the corresponding curve. The first Chern class of $\mathbb{L}_i$ is denoted $\psi_i$. 

Kapranov \cite{Kapranov} showed that each $\mathbb{L}_i$ is a base-point-free line bundle whose complete linear system induces a birational map $\overline{\mathcal{M}}_{0,n} \to \mathbb{P}^{n-3}$. For every subset $S$ of $[n] \coloneqq \{1, \dotsc, n\}$ of size at least $3$, we have a forgetful map $f_S \colon \overline{\mathcal{M}}_{0,n} \to \overline{\mathcal{M}}_{0,S} = \overline{\mathcal{M}}_{0,|S|}$. We therefore obtain a map
$$\overline{\mathcal{M}}_{0,n} \to \prod_{S \subset [n-1], \, |S| \ge 2} \mathbb{P}^{|S| - 2}$$
by composing $f_{S \cup n}$ with the map induced by the complete linear system of $\mathbb{L}_n$ on $\overline{\mathcal{M}}_{0,S \cup n}$. This map is a closed embedding.

Consider the braid arrangement $\mathcal{B}_{n-1}$ in $\mathbb{F}^{n -1}/\mathbb{F} \cdot (1, \dotsc, 1)$, whose hyperplanes are normal to $e_i - e_j$ for $i<j \in [n-1]$. The lattice of flats of $\mathcal{B}_{n-1}$ may be identified with the collection of partitions of the set $[n-1]$.  For any subset $S\subset[n-1]$ of cardinality at least 2,
let $F_S$ denote the flat corresponding to the partition of $[n-1]$ into $S$ along with a bunch of singletons.  
Projecting onto the factors indexed by flats of $\mathcal{B}_{n-1}$ of the form $F_S$ for some $S$, we have a map
$$W_{\mathcal{B}_{n-1}} \to \prod_{S \subset [n-1], \,  |S| \ge 2} \mathbb{P}(L^{F_S})$$
whose image is $\overline{\mathcal{M}}_{0,n}$ under the embedding described previously \cite[Section 4.3]{dCP95}. The relation between $\overline{\mathcal{M}}_{0,n}$ and $W_{\mathcal{B}_{n-1}}$ allows us to deduce results about $\overline{\mathcal{M}}_{0,n}$ from our study of wonderful varieties.

Let $\cL_S = f_{S \cup n}^* \mathbb{L}_n$.  This bundle is trivial when $|S|=2$, and the bundles corresponding to sets of cardinality at least 3 form a basis for the Picard group of $\overline{\calM}_{0, n}$. Let $$c_S \coloneqq c_1(\cL_S) = f_{S \cup n}^* \psi_n \in A^1(\overline{\calM}_{0, n}).$$

\begin{theorem} \label{thm:mzeron}
There is an isomorphism $\zeta_n\colon K(\overline{\mathcal{M}}_{0,n}) \to A(\overline{\mathcal{M}}_{0,n})$
that sends 
$1- [\cL_S^{-1}]$ to $c_S$.
For any $\xi\in K(\overline{\mathcal{M}}_{0,n})$, the Euler characteristic of $\xi$ is equal to the degree of 
$\frac{\zeta_n(\xi)}{1-c_{[n-1]}}$.
\end{theorem}

\subsection{Matroids}\label{sec:intro-matroids}
The presentations of $A(\WA)$ and $A(\Waug)$ depend only on the matroid associated with $\cA$, which
led to the definitions of Chow rings and augmented Chow rings of arbitrary 
loopless matroids \cite{FY,BHMPW20a}.
Similarly, our descriptions of $K(\WA)$ and $K(\Waug)$, along with their 
Euler characteristic functionals,
only depend on the matroid associated with $\cA$. With the exception of the latter part of Section \ref{sec:FY Snap}, we will always assume that all matroids are loopless.

Let $M$ be a matroid on the ground set $E$.  Let $\Sigma_M$ be the {\bf Bergman fan} of $M$, which has rays $\{\rho_F\mid\text{$F$ a nonempty proper flat}\}$ \cite{ardila2006bergman}.
Let $\Sigma^\aug_M$ be the
{\bf augmented Bergman fan} \cite{BHMPW20a} of $M$, which has rays
$\{\rho_F^\aug\mid \text{$F$ a proper flat}\}\cup\{\rho_e^\aug\mid e\in E\}$.  These two fans are related by the fact
that $\Sigma_M$ is isomorphic to the star of the ray $\rho_\varnothing$ in $\Sigma_M^\aug$ 
\cite[Proposition 2.7(2)]{BHMPW20a}.
The rings $A(M)$ and $A^\aug(M)$ are defined to be the Chow rings of the toric varieties
$X_{\Sigma_M}$ and $X_{\Sigma^\aug_M}$, respectively.  

If $M$ is the matroid associated with a hyperplane arrangement $\cA$, then, after choosing a linear functional defining each hyperplane in $\mathcal{A}$, we obtain canonical inclusions
$\WA\subset X_{\Sigma_M}$ and  $\Waug\subset X_{\Sigma^\aug_M}$
with the properties that $D_F = \WA \cap D_{\rho_F}$ for any nonempty proper flat $F$,
$D_F^\aug = \Waug\cap D_{\rho_F^\aug}$ for any proper flat $F$, and $D_e^\aug = \Waug\cap D_{\rho_e^\aug}$ for any $e\in E$.
The restriction maps $A(M) \to  A(\WA)$ and $A^\aug(M) \to A(\Waug)$ are both isomorphisms \cite{FY,BHMPW20a}.

Consider the polynomial ring
$T_M \coloneqq \Z[x_F\mid \text{$F$ a flat}] \otimes \Z[y_e\mid e\in E]$
along with the homomorphism $\PhFY^\aug\colon T_M\to A^\aug(M)$ taking $x_F$ to $[D_{\rho_F^\aug}]$ for all proper flats $F$,
$x_E$ to $-\sum_{F\neq E}[D_{\rho_F^\aug}]$, and
$y_e$ to $[D_{\rho_e^\aug}]$.
Composing with the pullback along the inclusion
$X_{\Sigma_M} \cong D_{\rho_\varnothing^\aug}\subset X_{\Sigma_M^\aug}$,
we obtain a homomorphism $\PhFY\colon T_M\to A(M)$ that sends $x_F$ to $D_{\rho_F}$ for all nonempty proper $F$ and $y_e$ to zero.
Consider also the polynomial ring $S_M \coloneqq \Z[u_F\mid \text{$F$ a nonempty flat}].$
We define $\Phsim\colon S_M\to A(M)$ and  $\Phsim^\aug\colon S_M\to A^\aug(M)$
via the formulas
\begin{equation}\label{add}\Phsim(u_F) \coloneqq -\sum_{F\subset G}\PhFY(x_G)\and \Phsim^\aug(u_F) \coloneqq -\sum_{F\subset G}\PhFY^\aug(x_G).\end{equation}

\begin{remark}\label{Asame}
When $M$ is the matroid associated with 
the arrangement $\cA$, the following diagrams (along with their augmented analogues) commute, thus justifying our repeated use
of the notation $\PhFY$ and $\Phsim$:
\[
\begin{tikzcd}
    T_M \ar[r, "="] \ar[d, "\PhFY"] & T_\cA \ar[d, "\PhFY"] && S_M \ar[r, "="] \ar[d, "\Phsim"] & S_\cA \ar[d, "\Phsim"] \\
    A(M) \ar[r, "\cong"] & A(\WA) && A(M) \ar[r, "\cong"] & A(\WA).
\end{tikzcd}
\]
Commutativity of the first diagram is immediate from the definitions, while commutativity of the second diagram 
is proved by reducing to the case of the Boolean matroid \cite[Section 3.2]{BES}.
\end{remark}

We define $K(M)$ and $K^\aug(M)$ to be the Grothendieck rings of vector bundles on $X_{\Sigma_M}$ and $X_{\Sigma_{M}^\aug}$, respectively. Because $X_{\Sigma_M}$ and $X_{\Sigma_M^{\aug}}$ are smooth, these coincide with the Grothendieck groups of coherent sheaves. 

\begin{proposition}\label{K is combinatorial}
If $M$ is the matroid associated with $\cA$, then the inclusions $\WA\subset X_{\Sigma_M}$ and $\Waug\subset X_{\Sigma_{M}^\aug}$ induce isomorphisms $K(M)\cong K(\WA)$ and $K^\aug(M)\cong K(\Waug)$.
\end{proposition}

Let $\cL_E^\aug$ be the line bundle on $X_{\Sigma_{M}^\aug}$ whose first Chern class is equal to $\sum_{F\neq E}[D_{\rho_F^\aug}]$.
We define the map
$\PsFY^\aug\colon T_M\to K^\aug(M)$
by sending $x_F$ to $[\cO_{D_{\rho_F^\aug}}]$ for all proper $F$, $x_E$ to $1-[\cL_E^\aug]$,
and $y_e$ to $[\cO_{D_{\rho_e^\aug}}]$ for all $e\in E$.
Composing with the pullback along the inclusion
$X_{\Sigma_M} \cong D_{\rho_\varnothing^\aug}\subset X_{\Sigma_M^\aug}$,
we obtain a homomorphism $\PsFY\colon T_M\to K(M)$ that sends $x_F$ to $[\cO_{D_{\rho_F}}]$ for all nonempty proper $F$ and $y_e$ to zero. We compute the kernel of $\PsFY$ and $\PsFY^{\aug}$ (Theorem~\ref{Kmatroid}), generalizing Theorem~\ref{FYK} and giving a presentation for $K(M)$ and $K^{\aug}(M)$. 

\begin{remark}\label{Ksame}
When $M$ is the matroid associated with 
the arrangement $\cA$, the following diagram commutes (as does its augmented analogue), thus justifying the repeated use
of the notation $\PsFY$:
\[
\begin{tikzcd}
    T_M \ar[r, "="] \ar[d, "\PsFY"] & T_\cA \ar[d, "\PsFY"] \\
    K(M) \ar[r, "\cong"] & K(\WA).
\end{tikzcd}
\]
\end{remark}

We also have the following extension of Theorem~\ref{exceptional} to matroids. 
\begin{theorem}\label{thm:exceptionalmatroid}
There exist isomorphisms $\zeta_M \colon K(M) \to A(M)$ and $\zeta^{\aug}_M \colon K^{\aug}(M) \to A^{\aug}(M)$, characterized by the property that
$$\zeta_M^{-1}(h_F) = 1 - \prod_{F\subset G} \left(1-\PsFY(x_F)\right)^{-1} \and (\zeta_M^{\aug})^{-1}(h_F^{\aug}) = 1 - \prod_{F\subset G} \left(1-\PsFY^\aug(x_F)\right)^{-1}.$$
\end{theorem}

Motivated by the above characterization of $\zeta_{\mathcal{A}}$ and $\zeta_{\mathcal{A}}^{\aug}$, we define maps $\Pssim\colon S_M\to K(M)$ and $\Pssim^\aug\colon S_M\to K^\aug(M)$ by the formulas
\begin{equation}\label{mult}\Pssim(u_F) = 1 - \prod_{F\subset G} \left(1-\PsFY(x_G)\right)^{-1} \and
\Pssim^\aug(u_F) = 1 - \prod_{F\subset G} \left(1-\PsFY^\aug(x_G)\right)^{-1}.\end{equation}
Then Theorem~\ref{thm:exceptionalmatroid} implies that the kernels of $\Pssim$ and $\Pssim^{\aug}$ are the same as the kernels of $\Phsim$ and $\Phsim^{\aug}$, which is computed by Theorem~\ref{Amatroid}. This gives a presentation of $K(M)$ and $K^{\aug}(M)$, which we call the \textbf{simplicial presentation}. 

\begin{remark}
Let $M$ be the matroid associated with an arrangement $\cA$.
Identifying $K(M)$ with $K(\WA)$ via the isomorphism in Proposition~\ref{K is combinatorial}, we have constructed two isomorphisms $\zeta_M, \zeta_{\mathcal{A}} \colon K(M) \to A(M)$. Proposition~\ref{prop:commute} implies that these two isomorphisms coincide, as do the analogous pair of isomorphisms in the augmented setting.
Similarly, we have constructed two maps $\Pssim \colon S_M \to K(M)$, one defined geometrically in Section \ref{sec:intro-simplicial}, and the other defined
algebraically in Equation \eqref{mult}. Proposition~\ref{prop:commute} also implies that these maps coincide, as do the analogous pair of maps in the augmented setting.
\end{remark}



The Chow rings $A(M)$ and $A^\aug(M)$ satisfy {\bf Poincar\'e duality}, meaning that there are maps $$\deg_M\colon A(M)\to\Z\and \deg^\aug_M\colon A^\aug(M)\to\Z$$
such that 
the pairing $A(M)\otimes A(M)\to \Z$ taking $f\otimes g$ to $\deg_M(fg)$ is perfect \cite[Theorem 6.19]{AHK},
and similarly in the augmented setting \cite[Theorem 1.3(4)]{BHMPW20a}. Since the toric varieties $X_{\Sigma_M}$ and $X_{\Sigma^{\aug}_M}$ are usually not proper, one cannot define an Euler characteristic map on $K(M)$ by pushing forward to a point. 
We define maps $$\chi\colon K(M)\to\Z\and \chi^\aug\colon K^{\aug}(M)\to\Z$$
by putting $$\chi(M, \xi)\coloneqq \deg_M\left(\frac{\zeta_M(\xi)}{1-h_E}\right) \and \chi^\aug(M, \xi^\aug)\coloneqq \deg_M^\aug\left(\frac{\zeta_M^\aug(\xi^\aug)}{1-h_E^\aug}\right).$$
Then the analogous pairing $K(M)\otimes K(M)\to \Z$ taking $\xi\otimes \eta$ to $\chi(M, \xi\eta)$ is perfect,
and similarly in the augmented setting.  For realizable matroids, this is a general property 
of the Euler pairing on the $K$-theory of a smooth proper linear variety \cite[Theorem 1.3]{anderson2015operational}.

\begin{remark}
When $M$ is not realizable, the maps $\chi$ and $\chi^\aug$ still have a natural geometric interpretation as Euler characteristics
on the wonderful variety (respectively augmented wonderful variety)
of the Boolean arrangement.  See Proposition \ref{prop:compatibility}.
\end{remark}

Although the $K$-rings $K(M)$ and $K^\aug(M)$ are isomorphic to their respective Chow rings,
the fact that they are $K$-rings endow them with several additional structures, including the structure of augmented $\la$-rings
(see Section~\ref{ssec:structure} for details). Furthermore, when $M$ is realizable, Serre duality gives a nontrivial identity satisfied by $\chi$ and $\chi^\aug$. In Theorem~\ref{thm:SD}, we show that that this identity extends to all matroids.

\subsection{Applications to Euler characteristics}\label{sec:euler-and-snap}
Much is known about intersection theory on (augmented) wonderful varieties. The Hirzebruch--Riemann--Roch-type formulas in Theorem~\ref{exceptional} allow us to transfer intersection-theoretic computations to $K$-theory. As applications, we give formulas for the Euler characteristic of the class of any line bundle in $K(M)$ expressed in terms of  the bundles $\cL_F$ (Theorem \ref{thm:euler-simplicial-matroid} and Corollary \ref{cor:snapper-simplicial-matroid}), and similarly in the augmented setting.
We also give a formula for the Euler characteristic of any line bundle expressed in terms of the 
Feichtner--Yuzvinsky generators of $K(M)$ (Theorem \ref{thm:snapper-FY-wonderful}), extending as a corollary 
Eur's formula \cite{Eur20} for the degrees of monomials in the Feichtner--Yuzvinsky generators of $A(M)$ (Corollary \ref{cor:volum-FY}).
Likewise, we give a formula for the Euler characteristic of any line bundle on $\overline{\mathcal{M}}_{0,n}$ expressed in terms the line bundles $\{\mathcal{L}_S\}$ (Theorem~\ref{thm:mzeron-snapper}). We also study the Euler characteristics of tensor products of the $\{\mathbb{L}_i\}$, giving a short proof of a theorem of Lee \cite{Lee} (Theorem~\ref{thm:psi-classes}).

\subsection{Structure of the paper}
In Sections~\ref{sec:FY} and \ref{sec:simplicial}, we compute a presentation of $K(\WA)$ and $K(\Waug)$ (Theorem \ref{FYK}) and prove
Theorem~\ref{exceptional}, assuming the following technical lemma:

\begin{lemma}\label{key}
The $K$-rings of $\WA$ and $\Waug$ are generated by the classes of line bundles.
\end{lemma}

In Section~\ref{sec:m0n}, we discuss the presentation of $K(\overline{\mathcal{M}}_{0,n})$ and the exceptional isomorphism to $A(\overline{\mathcal{M}}_{0,n})$ in Theorem \ref{thm:mzeron}.
In Section~\ref{sec:matroidk}, we give a presentation for $K(M)$ and $K^{\aug}(M)$ (Theorem~\ref{Kmatroid}), and we derive Lemma~\ref{key}
as a corollary. 
We have chosen to order our presentation in this way, rather than beginning with arbitrary matroids, 
in order better to highlight the geometric ideas underlying the proofs of our main results.

Section \ref{ssec:structure} is devoted to a description of the Adams operations in $K(M)$ and $K^\aug(M)$ along with a combinatorial Serre duality theorem (Theorem \ref{thm:SD}).
Section \ref{sec:applications} is dedicated to computing Euler characteristics in the simplicial generators, Section \ref{sec:FY Snap} to Euler characteristics in the Feichtner--Yuzvinsky generators, and Section \ref{mzeron snap} to Euler characteristics on $\overline{\mathcal{M}}_{0,n}$.
Finally, we include an appendix with a proof of the simplicial presentation for the (augmented) Chow ring of a matroid, which has not previously appeared in the literature.

We note that some of the proofs in Sections \ref{ssec:structure} and \ref{sec:FY Snap} proceed by first giving geometric proofs for realizable matroids, and then using the notion of
valuativity to extend the results to all matroids.  Thus, in these sections, separating the statements for hyperplane arrangements from the statements for matroids 
is absolutely essential.

\subsection{Acknowledgements}
We thank Chris Eur for explaining the simplicial generators of the augmented Chow ring to us, for suggesting that the simplicial generators should behave well with respect to the exceptional isomorphisms, and for other helpful conversations. We thank June Huh and Dhruv Ranganathan for helpful conversations. We thank David Speyer and the referee for helpful comments and corrections. 
This collaboration is supported by the NSF FRG grant DMS--2053261. The work of ML is also supported by an NDSEG fellowship, the work of SL is also supported by NSF DMS--1844768 and a Coline M. Makepeace Fellowship from Brown University, the work of SP is also supported by NSF DMS--2001502, and  the work of NP is also supported by NSF DMS--1954050.

\section{Feichtner--Yuzvinsky presentations} \label{sec:FY}
In this section, we analyze the Feichtner--Yuzvinsky presentation of $K(\WA)$ and $K(\Waug)$. We begin by recalling the Feichtner--Yuzvinsky presentation of $A(\WA)$ and $A(\Waug)$. 
We define the following ideals in $T_\cA$:
\begin{equation} \label{eq:Idef}
\begin{aligned}
\cI_1 &\coloneqq \Big\langle\sum_F x_F\Big\rangle\\
\cI_2 &\coloneqq \Big\langle y_e - \sum_{e\notin F} x_F\;\Big{|}\; e\in E\Big\rangle\\
\cI_3 &\coloneqq \langle x_Fx_G\mid \text{$F$ and $G$ incomparable}\rangle\\
\cI_4 &\coloneqq \langle y_e\mid e\in E\rangle\\
\cI_4^\aug &\coloneqq \langle y_ex_F\mid e\notin F\rangle.
\end{aligned}
\end{equation}

\begin{theorem}\label{FYChow}
Let $\cA$ be a hyperplane arrangement.
\begin{itemize}
\item[(1)]  
The map $\PhFY$ is surjective with kernel $\cI_1+\cI_2+\cI_3+\cI_4$ \cite[Corollary 2]{FY}.
\item[(2)]  
The map $\PhFY^\aug$ is surjective with kernel $\cI_1+\cI_2+\cI_3+\cI_4^\aug$ \cite[Remark 2.13]{BHMPW20a}.
\end{itemize}
\end{theorem}

\begin{remark}
The original Feichtner--Yuzvinsky presentation of $A(\WA)$ does not include the variable $x_\varnothing$.
Many later sources omit both $x_\varnothing$ and $x_E$, since the remaining $x_F$ are sufficient 
to generate and correspond to boundary divisors on $\WA$ (or correspond to rays of the Bergman fan \cite{ardila2006bergman} of the matroid represented by $\calA$).  In \cite{AHK},
$-\PhFY(x_E)$ is called $\alpha$ and $-\PhFY(x_\varnothing)$ is called $\beta$.  Similarly, the presentation of $A(\Waug)$ in 
\cite{BHMPW20a} does not include the variable $x_E$.
\end{remark}

We now describe the kernels of $\PsFY$ and $\PsFY^{\aug}$, giving a presentation for $K(\WA)$ and $K(\Waug)$. 
We define the following ideals in $T_\cA$:
\begin{equation} \label{eq:I'def}
\begin{aligned}
\cI_1' &\coloneqq \Big\langle 1 - \prod_F (1-x_F)\Big\rangle\\
\cI_2' &\coloneqq \Big\langle -(1-y_e) + \prod_{e\notin F} (1-x_F)\;\Big{|}\; e\in E\Big\rangle.
\end{aligned}
\end{equation}
Note that the generators of $\cI_1'$ and $\cI_2'$ are inhomogeneous, and their lowest order terms coincide with the generators of $\cI_1$ and $\cI_2$.
We give the following $K$-theoretic analogue of Theorem \ref{FYChow}.

\begin{theorem}\label{FYK}
Let $\cA$ be a hyperplane arrangement.
\begin{itemize}
\item[(1)]  
The map $\PsFY$ is surjective with kernel $\cI_1'+\cI'_2+\cI_3+\cI_4$.
\item[(2)]  
The map $\PsFY^\aug$ is surjective with kernel $\cI_1'+\cI'_2+\cI_3+\cI_4^\aug$.
\end{itemize}
\end{theorem}

\begin{remark}
The difference between the presentations in Theorems \ref{FYChow} and \ref{FYK} is that the homogeneous ideals $\cI_1$ and $\cI_2$ are replaced by the inhomogeneous ideals 
$\cI_1'$ and $\cI_2'$.  Indeed, the Chow rings of these varieties are isomorphic to the associated graded of the coniveau
filtrations on the $K$-rings.
\end{remark}

Before proving Theorem \ref{FYK}, we give general lemmas about the $K$-ring of a general smooth variety $X$,
which we will also apply later in the proofs of Theorems \ref{exceptional} and the analogue of Theorem~\ref{FYK} for general matroids (Theorem~\ref{Kmatroid}).

\begin{lemma}\label{tf}
If $A(X)$ is a free abelian group of finite rank $p$, then so is $K(X)$, and
the Chern character $\ch\colon K(X)\to A(X)_\Q$ is injective.
\end{lemma}

\begin{proof}
The Chern character becomes an isomorphism after tensoring with $\Q$ \cite[Example 15.2.16]{Fulton},
so $K(X)$ has rank $p$.  There is a surjective map from $A(X)$ to the associated graded of $K(X)$ with respect
to the coniveau filtration \cite[Example 15.1.5]{Fulton}.  Since $A(X)$ is free of rank $p$ and $K(X)$ also has rank $p$,
this implies that that $K(X)$ is free.  Finally, the Chern character factors as $K(X)\to K(X)_\Q\to A(X)_\Q$, with the first
map being injective by freeness of $K(X)$ and the second map being an isomorphism, so the Chern character is injective.
\end{proof}

\begin{remark}\label{when it happens}
Lemma \ref{tf} will be applied to the varieties $\WA$, $X_{\Sigma_M}$ \cite[Corollary 1]{FY} and $\overline{\calM}_{0, n}$ \cite[Corollary 2]{FY}, 
as well as to $\Waug$ and $X_{\Sigma_M^\aug}$ \cite[Theorem 6.19]{AHK}. 
\end{remark}

\begin{lemma}\label{generation}
Let $X$ be a smooth variety, and suppose that $K(X)$ is generated as a ring by the classes
of line bundles.  Let $D_1,\ldots,D_k$ be divisors on $X$.  If $A^1(X)$ is generated as an abelian group by $[D_1],\ldots,[D_k]$,
then $K(X)$ is generated as a ring by $[\cO_{D_1}],\ldots,[\cO_{D_k}]$.
\end{lemma}

\begin{proof}
Let $R$ be the subring of $K(X)$ generated by $[\cO_{D_1}],\ldots,[\cO_{D_k}]$.  We need to show that the class of
every line bundle is contained in $R$.  Since $[D_1],\ldots,[D_k]$ generate $A^1(X)$ as an abelian group, 
the line bundles $\cO(\pm D_1),\ldots,\cO(\pm D_k)$ generate the Picard group of $X$ under multiplication, so it will be sufficient to show that 
$[\cO(\pm D_i)]\in R$ for all $i$.

For any divisor $D$, we have an exact sequence $$0\to \cO(-D)\to \cO\to \cO_{D}\to 0,$$
which implies that \begin{equation*}\label{fundamental}[\cO(-D)] = [\cO] - [\cO_{D}] = 1 - [\cO_{D}].\end{equation*}
We also have
\begin{equation}\label{almost as fundamental}[\cO(D)] = [\cO(-D)]^{-1} = \frac{1}{1-[\cO_{D}]} = 1 + [\cO_{D}] + [\cO_{D}]^2 + \cdots.\end{equation}
Since $[\cO_{D}]$ lives in the first piece of the coniveau filtration on $K(X)$, it is nilpotent, so the sum terminates.
This allows us to conclude that both $[\cO(-D_i)]$ and $[\cO(D_i)]$ live in the ring $R$.
\end{proof}

\begin{proof}[Proof of Theorem \ref{FYK}.]
We begin with statement (2).  Surjectivity of the map $\PhFY^\aug$ follows from 
Theorem \ref{FYChow}, Lemma \ref{key}, and Lemma \ref{generation}.  
Next, we show that $\cI_1' + \cI_2' + \cI_3 + \cI_4^\aug$
is contained in the kernel of $\PhFY^\aug$.  The ideal $\cI_3$ is contained in the kernel because $D_F$ and $D_G$ are disjoint when $F$ and $G$
are incomparable.  Similarly, $\cI_4^\aug$ is contained in the kernel because $D_e$ is disjoint from $D_F$ whenever $e\notin F$.
To prove that $\cI_1'$ is contained in the kernel, we observe that its generator maps to
$$1 - [\cL_E^\aug]\prod_{F\neq E}\Big(1-[\cO_{D_F^\aug}]\Big) =
1 - [\cL_E^\aug]\prod_{F\neq E}[\cO(-D_F^\aug)].$$
Since $\cI_1$ is contained in the kernel of $\PhFY^\aug$, the line bundle
$$\cL_E^\aug\otimes\bigotimes_{F\neq E} \cO(-D_F^\aug)$$
is trivial, thus the generator of $\cI_1'$ is contained in the kernel of $\PsFY^\aug$.

To prove that $\cI_2'$ is contained in the kernel, Lemma \ref{tf} tells us that it is sufficient to prove that, for all $e\in E$,
we have
$$\ch\Big(1-[\cO_{D_e^\aug}]\Big) = \prod_{e\notin F} \ch\Big(1-[\cO_{D_F^\aug}]\Big).$$
We have
$$\ch\Big(1-[\cO_{D_e^\aug}]\Big) = \ch([\cO(-D_e^\aug)]) = \exp\Big(-[D_e^\aug]\Big)$$
and
$$\prod_{e\notin F} \ch\Big(1-[\cO_{D_F^\aug}]\Big) 
= \prod_{e\notin F} \exp\Big(-[D_F^\aug]\Big) = \exp\Big(-\sum_{e\notin F} [D_F^\aug]\Big).$$
The fact that these are equal follows from the fact that $\cI_2$ is contained in the kernel of $\PhFY^\aug$.

Let $R \coloneqq T_\cA\,\big{/}\,(\cI_1'+\cI_2' + \cI_3+\cI_4^\aug)$.  We have now shown that $R$ surjects onto $K(\Waug)$, and we need
to prove that the map is injective.  
Let $p$ be the rank of the free abelian group $A(\Waug)$.
Consider the decreasing filtration
$$R = F_0\supset F_1\supset\cdots,$$
where $F_i$ is the span of all monomials of total degree $\geq i$.
Since the leading terms of the generators $\cI_1'$ and $\cI_2'$ coincide with the generators of $\cI_1$ and $\cI_2$,
we have a surjection
$$A(\Waug) \cong T_\cA\,\big{/}\,(\cI_1+\cI_2 + \cI_3+\cI_4^\aug)\to\gr R.$$
In particular, this implies that the abelian group $\gr R$ can be generated by $p$ elements, and so the same is true of $R$.
Lemma \ref{tf} tells us that $K(\Waug)$ is also free abelian of rank $p$, so any surjection from $R$ to $K(\Waug)$ must be an isomorphism.

The proof of statement (1) is nearly identical.  The only extra ingredient is the argument that $\cI_4$ is contained 
in the kernel of the map $\PsFY$, which follows from the fact that $D_e^\aug$ is disjoint from $D_\varnothing^\aug\cong \WA$ inside of $\Waug$.
\end{proof}

\section{Exceptional isomorphisms} \label{sec:simplicial}

The purpose of this section is to prove Theorem~\ref{exceptional}.
We begin by observing that we have an inclusion
$$\WA\;\;\;\subset\prod_{\substack{\text{$F\subset E$}\\ \text{a nonempty flat}}} \mathbb{P}(L^F)\;\;\;
\subset \prod_{\substack{\text{$S\subset E$}\\ \text{a nonempty subset}}} \mathbb{P}(L^S).$$
The first inclusion comes from the definition of $\WA$, and the second  from the diagonal
embedding of $\mathbb{P}(L^F)$ into the product of $\mathbb{P}(L^S)$ for all $S$ such that $L^S = L^F$.\footnote{In matroid
theoretic language, this is the collection of subsets $S$ whose closure is $F$.}
Similarly, we have
$$\Waug\;\;\;\subset\prod_{\substack{\text{$F\subset E$}\\ \text{a nonempty flat}}} \mathbb{P}(L^F\oplus\mathbb{F})\;\;\;
\subset \prod_{\substack{\text{$S\subset E$}\\ \text{a nonempty subset}}} \mathbb{P}(L^S\oplus\mathbb{F}).$$

Suppose that $L'\subset L$ is a linear subspace.  Let $E' \coloneqq \{e\in E\mid L'\not\subset H_e\}$,
and define a new hyperplane arrangement $$\cA_{L'} \coloneqq \{H_e\cap L'\mid e\in E'\}.$$ 
If $E'=E$, then we have inclusions
$W_{\!\cA_{L'}}\subset\WA$ and 
$W^\aug_{\!\cA_{L'}}\subset\Waug$, each inside of the corresponding product of projective spaces indexed by subsets of $E$.
If $L'$ is contained in one or more hyperplane, then we still have an inclusion $W^\aug_{\!\cA_{L'}}\subset\Waug$, where
$W^\aug_{\!\cA_{L'}}$ sits inside the product indexed by subsets of $E$ by setting the $S$ coordinate to $0\in L^S \subset  \mathbb{P}(L^S\oplus\mathbb{F})$
unless $S\subset E'$.

For any flat $F$ of $\cA$, let ${F}' \coloneqq \{e\in E' \mid L_F\cap L'\subset H_e\}$
be the corresponding flat of $L'$.  Then the line bundle $\cL_F^{\aug}$ on $\Waug$ restricts to the line bundle $\cL_{ F'}$ 
on $W_{\!\cA_{L'}}^{\aug}$.  If $E = E'$, then the line bundle $\cL_F$ on $\WA$ restricts to the line bundle $\cL_{ F'}$ 
on $W_{\!\cA_{L'}}$.

\begin{lemma}\label{lem:generic}
Fix a flat $F$, and
suppose that $H\subset L$ is a hyperplane with the property that $L_F\subset H$ and 
$L_G\not\subset H$ for all flats $G\subsetneq F$.
Then $$h^\aug_F 
= [W^\aug_{\!\cA_H}] \and 
\eta^\aug_F 
= [\cO_{W^\aug_{\!\cA_H}}].$$
If the rank of $F$ is greater than 1, then
$$h_F 
= \left[W_{\!\cA_H}\right] \and \eta_F 
= \left[\cO_{W_{\!\cA_H}}\right].$$
\end{lemma}

\begin{proof}
We prove only the first statement; the second is similar.
Let $g$ be a section of $\cO_{\mathbb{P}(L^F\oplus \mathbb{F})}(1)$ with vanishing locus equal to
$\mathbb{P}(H/L_F\oplus \mathbb{F}) \subset \mathbb{P}(L^F\oplus \mathbb{F})$.
Then $(\pi_F^\aug)^*g$ is a regular section of $\cL_F^\aug$ with vanishing locus $W^\aug_{\!\cA_H}$, which shows that 
$h_F^\aug = c_1(\cL_F^\aug) = [W^\aug_{\!\cA_H}]$.
We interpret $(\pi_F^\aug)^*g$ as an element
of $\Hom(\cO_{\Waug},\cL_F^\aug)\cong \Hom\big((\cL_F^\aug)^{-1},\cO_{\Waug}\big)$ to obtain 
an exact sequence
$$0\to(\cL_F^\aug)^{-1}\overset{(\pi^{\aug}_F)^*g}{\longrightarrow}\cO_{\Waug}\to\cO_{W^\aug_{\!\cA_H}}\to 0,$$
which shows that $\eta^\aug_F = 1 - [(\cL_F^\aug)^{-1}] = [\cO_{\Waug}] -  [(\cL_F^\aug)^{-1}] =  [\cO_{W^\aug_{\!\cA_H}}]$. 
\end{proof}

For any tuple of natural numbers $\textbf{m} = (m_F \mid F \text{ a nonempty flat})$, we define the monomials
\begin{equation*}
\begin{aligned}
h^{\textbf{m}} &\coloneqq \prod_{F \not= \emptyset} h_F^{m_F}\in A(\WA) & \quad \quad 
(h^\aug)^{\textbf{m}} &\coloneqq \prod_{F \not= \emptyset} (h_F^\aug)^{m_F}\in A(\Waug)\\
\eta^{\textbf{m}} &\coloneqq \prod_{F \not= \emptyset} \eta_F^{m_F} \in K(\WA) & 
(\eta^\aug)^{\textbf{m}} &\coloneqq \prod_{F \not= \emptyset} (\eta_F^\aug)^{m_F}\in K(\Waug).
\end{aligned}
\end{equation*}

\begin{lemma}\label{lem:subspace}
Suppose that $\mathbb{F}$ is infinite.
For any $\textbf{m}$, one of the following two statements holds:
\begin{itemize}
\item $\eta^{\textbf{m}} = 0$ and $h^{\textbf{m}}=0$.
\item There exists a linear subspace $L'\subset L$, not contained in any of the hyperplanes in $\cA$,
such that $\eta^{\textbf{m}} = [\mathcal{O}_{W_{\!\cA_{L'}}}]$ and $h^{\textbf{m}} = [W_{\!\cA_{L'}}]$.
\end{itemize}
In addition, one of the following  two statements holds:
\begin{itemize}
\item $(\eta^\aug)^{\textbf{m}} = 0$ and $(h^\aug)^{\textbf{m}}=0$.
\item There exists a linear subspace $L'\subset L$
such that $(\eta^\aug)^{\textbf{m}} = [\mathcal{O}_{W^\aug_{\!\cA_{L'}}}]$ and $(h^\aug)^{\textbf{m}} = [W^\aug_{\!\cA_{L'}}]$.
\end{itemize}
\end{lemma}

\begin{proof}
We will prove only the first claim; the augmented case is similar. 
We proceed by induction on $\sum m_F$.  When $\textbf{m}=0$, we can take $L'=L$.  For the inductive step, assume that the second
statement holds for $\textbf{m}'$, and take $\textbf{m}$ such that $m_F = m'_F + 1$ and $\textbf{m}$ agrees with $\textbf{m}'$ in every other coordinate.
Then
$$\eta^{\textbf{m}} = \eta_F \eta^{\textbf{m}'} = \eta_F [\mathcal{O}_{W_{\!\cA_{L'}}}] = \big(1-[\cL_F^{-1}]\big)[\mathcal{O}_{W_{\!\cA_{L'}}}]  
= [\mathcal{O}_{W_{\!\cA_{L'}}}] - [\mathcal{L}_{F'}^{-1}],$$
where $F' =\{e\mid L_F\cap L'\subset H_e\}$ and $\cL_{F'}$ is the corresponding line bundle on $W_{\!\cA_{L'}}$.
If $F'$ has rank 1, then $\cL_{F'}$ is trivial, and we get zero.  If $F'$ has rank greater than 1, then we may choose a 
hyperplane\footnote{This is where we use the fact that $\mathbb{F}$ is infinite.}
$H'\subset L'$ such that $H'\notin\cA'$, $L_{F'}\subset H'$, and $L_{G'}\not\subset H'$ for all flats $G'\subsetneq F'$.
Then by Lemma \ref{lem:generic}, $\eta^{\textbf{m}} = [\mathcal{O}_{W_{\!\cA_{L'}}}] - [\mathcal{L}_{F'}^{-1}] = [\mathcal{O}_{W_{\!\cA_{H'}}}]$.
Similarly, we have
$$h^{\textbf{m}} = h_F h^{\textbf{m}'} = h_F[W_{\!\cA_{L'}}] = c_1(\cL_F) [W_{\!\cA_{L'}}].$$
By the projection formula, this is equal to the pushforward of $c_1(\cL_{F'})$ from $W_{\!\cA_{L'}}$ to $\WA$.
If $F'$ has rank 1, then $\cL_{F'}$ is trivial, and we get zero.  If $F'$ has rank greater than 1, then we may choose $H'$
as above, and we get $[W_{\!\cA_{H'}}]$ by Lemma \ref{lem:generic}.
\end{proof}

\begin{remark}\label{bumping the last coordinate}
It is useful to think about the special case when $F=E$ in the proof of Lemma \ref{lem:subspace}.  In this case,
$F'=E$, and the rank of $F'$ is equal to the dimension of $L'$.  This means that, if $h^{\textbf{m}} = [W_{\!\cA_{L'}}]$ for some
$L'$ of dimension greater than 1, then $h^{\textbf{m}}h_E = [W_{\!\cA_{H'}}]$ for some generic hyperplane $H'\subset L'$.
Iterating this observation, we see that, if $k= \dim L - 1 - \sum m_F$, then $h^{\textbf{m}}h_E^k\in A^{\dim L - 1}(\WA)$ is the class of a point.
Similarly, $(h^\aug)^{\textbf{m}}(h_E^\aug)^{k+1}\in A^{\dim L}(\Waug)$ is the class of a point.
\end{remark}

\begin{proposition}\label{prop:HRRtype}
We have
\begin{equation}\label{eq:HRR}
\chi(\WA, \eta^{\textbf{m}}) = \deg_{\WA}\left( \frac{h^{\textbf{m}}}{1 - h_E}\right) \quad and \quad \chi(\Waug, (\eta^{\aug})^{\textbf{m}}) = \deg_{\Waug}\left( \frac{(h^{\aug})^{\textbf{m}}}{1 - h_E^{\aug}}\right).
\end{equation}
\end{proposition}

\begin{proof}
We will prove only the non-augmented case; the augmented case is identical. 
The Euler characteristic can be computed after extension of scalars, so we can assume that $\mathbb{F}$ is infinite.
By Lemma \ref{lem:subspace}, either $\eta^{\textbf{m}}=0$ and $h^{\textbf{m}}=0$, in which case Equation \eqref{eq:HRR} holds trivially,
or there exists some $L'\subset L$ such that $\eta^{\textbf{m}} = [\cO_{W_{\!\cA_{L'}}}]$ and $h^{\textbf{m}} = [W_{\!\cA_{L'}}]$.
Then $\chi(\WA, \eta^{\textbf{m}}) = 1$ because $W_{\!\cA_{L'}}$ is a smooth iterated blow-up of projective space, and $\deg_{\WA}\left( \frac{h^{\textbf{m}}}{1 - h_E}\right) = 1$ by Remark \ref{bumping the last coordinate}. 
\end{proof}

We are now ready to prove the existence of exceptional isomorphisms from $K$-rings to Chow rings for wonderful and augmented wonderful varieties.

\begin{proof}[Proof of Theorem~\ref{exceptional}]
Again, we prove only the non-augmented case.  Note that, once we prove the first sentence
in the theorem, the second sentence will follow from Proposition \ref{prop:HRRtype}.

For each nonempty flat $F$, we have an isomorphism $$\zeta_F\colon K(\mathbb{P}(L^F))\to A(\mathbb{P}(L^F))$$
sending the structure sheaf of a hyperplane (which we will denote by $\sigma_F$) to the Chow class of a hyperplane (which we
will denote by $s_F$).  This isomorphism has the property that,
for any $\xi_F\in K(\mathbb{P}(L^F))$, $$\chi(\mathbb{P}(L^F), \xi_F) = \deg_{\mathbb{P}(L^F)}\left(\frac{\zeta_F(\xi_F)}{1-s_F}\right).$$
By the K\"unneth formula \cite[Proposition 6.4]{anderson2015operational}, we have an isomorphism
$$\zeta = \otimes_F \zeta_F \colon  K\bigg(\prod_{F \not= \emptyset} \mathbb{P}(L^F)\bigg)\to A\bigg(\prod_{F \not= \emptyset} \mathbb{P}(L^F)\bigg)$$
that takes $\sigma_F$ to $s_F$ for every $F$, and has the property that, for any $\xi\in K\left(\prod_{F \not= \emptyset} \mathbb{P}(L^F)\right)$,
\begin{equation}\label{eq-euler-deg}\chi\Big(\prod_{F \not= \emptyset} \mathbb{P}(L^F), \xi\Big) = \deg_{\prod_{F \not= \emptyset} \mathbb{P}(L^F)} \left(\frac{ \zeta(\xi)}{\prod_{F\not= \emptyset} (1-s_F)}\right).\end{equation}

By \cite[Remark 3.2.6]{BES} or Theorem~\ref{Amatroid}, the restriction map 
\begin{eqnarray*} A\Big(\prod_{F \not= \emptyset} \mathbb{P}(L^F)\Big) \to A(\WA) \qquad \text{given by} \qquad s_F \mapsto h_F\end{eqnarray*}
is surjective. By Poincar\'{e} duality, its kernel is equal to the annihilator of $[\WA]$. 
By Lemma~\ref{key} and Lemma~\ref{generation}, the analogous restriction map 
\begin{eqnarray*} K\Big(\prod_{F \not= \emptyset} \mathbb{P}(L^F)\Big) \to K(\WA) \qquad \text{given by} \qquad
\sigma_F &\mapsto& \eta_F\end{eqnarray*}
is also surjective. By the nondegeneracy of the Euler pairing \cite[Proposition 6.3]{anderson2015operational}, the kernel is equal to the annihilator of $[\cO_{\WA}]$. 
We will prove that \begin{equation}\label{unit}\zeta([\cO_{\WA}]) = \prod_{F\neq \emptyset, E} (1-s_F)\cdot [\WA].\end{equation}
Since $s_F$ is nilpotent, the above product is a unit, which implies that $\zeta$ takes the annihilator of $[\cO_{\WA}]$
to the annihilator of $[\WA]$.  This in turn shows that $\zeta$
descends to an isomorphism $\zeta_\cA\colon K(\WA)\to A(\WA)$.

We now prove Equation \eqref{unit}.
For any $\textbf{m} = (m_F \mid F \text{ a nonempty flat})$, let $$s^{\textbf{m}} \coloneqq \prod_{F \not= \emptyset} s_F^{m_F}\and\sigma^{\textbf{m}} \coloneqq \prod_{F\not= \emptyset} \sigma_F^{m_F},$$ so that $s^{\textbf{m}}\mapsto h^{\textbf{m}}$ and $\sigma^{\textbf{m}}\mapsto \eta^{\textbf{m}}$ under restriction.
We have
\begin{eqnarray*}
\deg_{\prod_{F \not= \emptyset} \mathbb{P}(L^F)}\Big(s^{\textbf{m}}\cdot\zeta([\cO_{\WA}])\Big)
&=& \chi\Big(\prod_{F \not= \emptyset} \mathbb{P}(L^F), \sigma^{\textbf{m}} \cdot \prod_{F \not= \emptyset}(1-\sigma_F)\cdot [\cO_{\WA}]\Big)\qquad\text{by Equation \eqref{eq-euler-deg}}\\ 
&=& \chi\Big(\WA, \eta^{\textbf{m}}\cdot \prod_{F \not= \emptyset}(1-\eta_F)\Big)\\
&=& \deg_{\WA}\Big(h^{\textbf{m}}\cdot  \prod_{F\neq \emptyset, E}(1-h_F)\Big)\qquad\text{by Proposition \eqref{prop:HRRtype}}\\ 
&=& \deg_{\prod_{F \not= \emptyset} \mathbb{P}(L^F)}\Big(s^{\textbf{m}}\cdot  \prod_{F\neq \emptyset, E}(1-s_F)\cdot [\WA]\Big).
\end{eqnarray*}
Equation \eqref{unit} then follows from Poincar\'e duality.
\end{proof}

\section{The $K$-ring of $\overline{\mathcal{M}}_{0,n}$} \label{sec:m0n}
We now apply our results to study the $K$-ring of $\overline{\mathcal{M}}_{0,n}$, and 
in particular prove the exceptional isomorphism of Theorem~\ref{thm:mzeron}.
From Section~\ref{ssec:m0n}, we observe that the following diagram commutes:
\[
\begin{tikzcd}
    W_{\calB_{n-1}} \ar[r, hookrightarrow] \ar[d, "p"] & \displaystyle\prod_{F \not= \emptyset} \P(L^{F}) \ar[d, "pr"] \\
    \overline{\calM}_{0, n} \ar[r, hookrightarrow] & \displaystyle\prod_{S \subset [n-1], \, |S| \ge 2} \P(L^{F_S}),
\end{tikzcd}
\]
where $pr$ is the projection onto the factors indexed
by flats of the form $F_S$, and 
$p$ is the restriction of $pr$. By \cite[Theorem 4.2]{FeichtnerMuller}, $p$ is a composition of blow-ups at smooth centers. By \cite[Proposition 6.7(b)]{Fulton},
the pullback maps
$$\begin{aligned}
p^*\colon A(\overline{\calM}_{0, n}) \to A(W_{\calB_{n-1}}) \qquad \text{given by} \qquad
c_S \mapsto h_{F_S}, \qquad \text{and}
\end{aligned}$$
$$\begin{aligned}
p^*\colon K(\overline{\calM}_{0, n}) \to K(W_{\calB_{n-1}}) \qquad \text{given by} \qquad
1 - [\calL_S^{-1}] \mapsto \eta_{F_S}
\end{aligned}$$
are both injective. 

Consider the ring 
$R_n \coloneqq \Z[u_S \mid S \subset [n-1], |S| \ge 2 ] \subset S_{\calB_{n-1}},$
where the inclusion sends $u_S$ to $u_{F_S}$.
Then the map $S_{\calB_{n-1}}\to A(W_{\calB_{n-1}})$ given by the simplicial presentation
restricts to a map $R_n\to A(\overline{\calM}_{0, n})$.
Consider the following ideals in $R_n$: 
$$\begin{aligned}
\calK_1 &\coloneqq \left\langle (u_S - u_{S \cup T})(u_T - u_{S \cup T}) \mid S \cap T \neq \varnothing \right\rangle\\
\calK_2 &\coloneqq \langle u_{S} \mid |S| = 2\rangle.\\
\end{aligned}
$$

\begin{theorem}\label{ehkr}{\em \cite{singh} \cite[Theorem 5.5]{MR2630055}}
\label{mzeron}
    The map $R_n \to A(\overline{\calM}_{0, n})$ is surjective with kernel $\calK_1+\calK_2$. 
\end{theorem}

The surjectivity statement in Theorem \ref{ehkr} allows us to prove Theorem \ref{thm:mzeron}.

\begin{proof}[Proof of Theorem \ref{thm:mzeron}.]
It is known to experts that $K(\overline{\mathcal{M}}_{0,n})$ is generated by line bundles; see the discussion in the introduction of \cite{castravet2020derived}. This can also be proved in the same way as Lemma~\ref{key}. 
Therefore,  Theorem~\ref{ehkr} and Lemma~\ref{generation} imply that $K(\overline{\calM}_{0, n})$ is generated by
$1- [\calL_S^{-1}]$ for all subsets $S\subset[n-1]$
of cardinality at least 2.  By Theorems \ref{exceptional} and \ref{ehkr},
restriction of the isomorphism $\zeta_{\calB_{n-1}}\colon K(W_{\calB_{n-1}})\to A(W_{\calB_{n-1}})$
to $$K(\overline{\calM}_{0, n}) \subset K(W_{\calB_{n-1}})$$
takes $K(\overline{\mathcal{M}}_{0,n})$ isomorphically to $A(\overline{\calM}_{0, n}) \subset A(W_{\calB_{n-1}}).$
\end{proof}

\begin{remark}
    In \cite{CT2}, Castravet and Tevelev studied 
    $K(\overline{\calM}_{0, n})$
    as a representation of the symmetric group $S_n$; in particular, they show that the $S_n$-action on $K(\overline{\calM}_{0, n})$ is a permutation representation over $\Z$.
    The action of $S_n$ on $A(\overline{\mathcal{M}}_{0,n})$ has also been studied, beginning with \cite{Getzler}. Our isomorphism
    in Theorem \ref{thm:mzeron} is not $S_n$-equivariant, but it is equivariant
    with respect to the action of the subgroup $S_{n-1}\subset S_n$ that fixes $n$. Note that $K(\overline{\calM}_{0, n}) \otimes \mathbb{Q}$ and $A(\overline{\mathcal{M}}_{0,n}) \otimes \mathbb{Q}$ are $S_n$-equivariantly isomorphic via the Chern character, but one can check that $A(\overline{\mathcal{M}}_{0,5})$ is not a permutation representation of $S_5$ over $\Z$, and therefore cannot be $S_5$-equivariantly isomorphic to $K(\overline{\mathcal{M}}_{0,5})$. 
\end{remark}

Theorems \ref{thm:mzeron} and \ref{ehkr} combine to give us the following corollary.

\begin{corollary}
\label{thm:simplicial-k-moduli}
    The homomorphism $R_n \to K(\overline{\calM}_{0, n})$ taking $u_S$ to 
    $1- [\calL_S^{-1}]$ is surjective with 
    kernel $\calK_1+\calK_2$. 
\end{corollary}

\begin{remark}
The variety $\overline{\mathcal{M}}_{0,n}$ is a special case of a wonderful variety with a \emph{building set}. The proof of Theorem~\ref{thm:mzeron} generalizes to any wonderful variety with a building set, as the surjectivity part of Theorem~\ref{ehkr} can be proved for any wonderful variety with a building set along of the lines of \cite[Section 3.2]{BES}. Therefore, the application in this section can be extended to the Hassett compactification of \textit{heavy/light} weighted rational stable curves \cite{Hassett}. It is also a wonderful variety with a certain building set \cite{CHMR}, and its Chow ring has been studied in \cite{Kannan2021ChowRO}. 
\end{remark}

\section{Matroid $K$-rings} \label{sec:matroidk}
In this section, we study $K(M)$ and $K^{\aug}(M)$, the $K$-rings of $X_{\Sigma_M}$ and $X_{\Sigma_{M}^{\aug}}$, respectively. We begin by establishing some basic properties that show that $K(M)$ and $K^{\aug}(M)$ behave similarly to the $K$-ring of the wonderful variety and augmented wonderful variety of a realization. 
We will use the following notation:
\begin{equation*}
\begin{aligned}
t_F &\coloneqq \PhFY(x_F)\in A(M) & \quad \quad 
t_F^\aug &\coloneqq \PhFY^\aug(x_F)\in A^\aug(M)\\
\tau_F &\coloneqq \PsFY(x_F) \in K(M) & 
\tau_F^\aug &\coloneqq \PsFY^\aug(x_F) \in K^\aug(M)\\
\end{aligned}
\end{equation*}
By Equation \eqref{add}, we have
\begin{equation}\label{the simple direction}
h_F = -\sum_{F\subset G} t_G \and
h_F^\aug = -\sum_{F\subset G} t_G^\aug.
\end{equation}
We also recall the Feichtner--Yuzvinsky presentations of $A(M)$ and $A^{\aug}(M)$. Define the ideals $\cI_1,\cI_2,
\cI_3,\cI_4,\cI_4^\aug\subset T_M$ 
as in Equation \eqref{eq:Idef}.

\begin{theorem}\label{AmatroidFY}
Let $M$ be a loopless matroid.
\begin{itemize}
\item[(1)] The map 
$\PhFY$ is surjective with kernel $\cI_1+\cI_2+\cI_3+\cI_4$ \cite{FY}.
\item[(2)] The map 
$\PhFY^\aug$ is surjective with kernel $\cI_1+\cI_2+\cI_3+\cI_4^\aug$ \cite{BHMPW20a}.
\end{itemize}
\end{theorem}

\begin{proof}[Proof of Proposition \ref{K is combinatorial}.]
Since $A(M) = A(X_{\Sigma_M})$ is torsion-free, it is isomorphic to the associated graded of $K(M) = A(X_{\Sigma_M})$ with
respect to the coniveau filtration \cite[Example 15.2.16]{Ful93}.
Since we know that the associated graded map $A(M)\to A(\WA)$ is an isomorphism, the filtered map $K(M)\to K(\WA)$
is also an isomorphism \cite[Theorem 5.2.12]{MR1269324}.  The augmented case is identical.
\end{proof}

\begin{proof}[Proof of Lemma \ref{key}.]
By Proposition \ref{K is combinatorial}, this is equivalent to the statement that the rings
$K(M)$ and $K^\aug(M)$ are generated by line bundles on the toric varieties $X_{\Sigma_M}$ and $X_{\Sigma_M^\aug}$.
This is a general property of smooth toric varieties, see \cite[Lemma 2.2]{anderson2015operational}.
\end{proof}

We have the following generalization of Theorem~\ref{FYK} to matroids. 

\begin{theorem}\label{Kmatroid}
Let $M$ be a loopless matroid.
\begin{itemize}
\item[(1)] The map 
$\PsFY$ is surjective with kernel $\cI'_1+\cI'_2+\cI_3+\cI_4$.
\item[(2)] The map 
$\PsFY^\aug$ is surjective with kernel $\cI'_1+\cI'_2+\cI_3+\cI_4^\aug$.
\end{itemize}
\end{theorem}

\begin{proof}
The proof of statements (1) and (2) is essentially identical to the proof of Theorem \ref{FYK}, except that we use
Theorem \ref{AmatroidFY} in place of Theorem \ref{FYChow}, and we use the fact that the $K$-ring of a smooth toric variety is generated by classes of line bundles (see \cite[Lemma 2.2]{anderson2015operational}) in place of Lemma \ref{key}.
\end{proof}

We now compute the simplicial presentation of $K(M)$ and $K^{\aug}(M)$, and then use that to prove Theorem~\ref{thm:exceptionalmatroid}. We first state the simplicial presentation of $A(\WA)$ and $A(\Waug)$, which describes the kernel of $\Phsim$ and $\Phsim^{\aug}$. We defer the proof to Appendix~\ref{sec:augChow}.

For any pair of flats $F$ and $G$,
let $F \vee G$ denote the smallest flat containing both $F$ and $G$. 
Consider the following ideals in $S_\cA$:

\begin{equation} \label{eq:Jdef}
\begin{aligned}
\cJ_1 &\coloneqq \langle (u_F - u_{F \vee G})(u_G - u_{F \vee G}) \mid \textrm{$F$, $G$ arbitrary} \rangle\\
\cJ_2 &\coloneqq \langle u_F\mid \rk F = 1 \rangle \\
\cJ_2^\aug &\coloneqq \langle u_F^2\mid \text{$\rk F = 1$} \rangle + \langle u_F(u_G - u_{F\vee G})\mid \text{$\rk F = 1$, $G$ arbitrary}\rangle.
\end{aligned}
\end{equation}

\begin{theorem}\label{Amatroid} 
Let $M$ be a loopless matroid.
\begin{itemize}
\item[(1)] The map 
$\Phsim$ is surjective with kernel $\cJ_1+\cJ_2$.
\item[(2)] The map 
$\Phsim^\aug$ is surjective with kernel $\cJ_1+\cJ_2^\aug$.
\end{itemize}
\end{theorem}

\begin{proof}[Proof of Theorem \ref{thm:exceptionalmatroid}]
We prove only the augmented case; the non-augmented case is similar.
Define a map $\kappa_M\colon  S_M\to K^\aug(M)$ by putting 
$$\kappa_M(u_F) =  1 - \prod_{F\subset G} (1-\tau^\aug_G)^{-1} = 1 - \prod_{F\subset G} (1 + \tau^\aug_G + (\tau^\aug_G)^2 + \cdots).$$
Our first task will be to show that $\kappa$ vanishes on $\cJ_1+\cJ_2^\aug$, and therefore descends to a map from $A^\aug(M)$ to $K^\aug(M)$.
We will make use of Theorem \ref{Kmatroid}(2), which says that
$\tau^\aug_F\tau^\aug_G = 0$ for any incomparable $F$ and $G$, 
\begin{equation*}\label{eq:teq1} \prod_F (1-\tau^\aug_F) = 1, \qquad \text{and} \qquad \Psi_{FY}^\aug(y_e) = 1 - \prod_{e\notin G} (1 - \tau_G^\aug).\end{equation*}
The later equation implies that $K^{\aug}(M)$ is generated by the $\tau_G^{\aug}$. Using the relations in $\mathcal{I}_4^{\aug}$, we get that
\begin{equation}\label{eq:teq2} \tau^\aug_F\left(1-\prod_{e\notin G}(1-\tau^\aug_G)\right) = 0\end{equation}
for any $e\notin F$.

We begin by checking that $\kappa_M$ vanishes on a generator of $\cJ_1$.  We have
\begin{eqnarray*}
&& \kappa_M\Big((u_F-u_{F\vee G})(u_G-u_{F\vee G})\Big)\\
&=& \left(\prod_{F\subset H}\left(1-\tau^\aug_H\right)^{-1} - \prod_{F\vee G\subset I}\left(1-\tau^\aug_I\right)^{-1}\right)
\left(\prod_{G\subset J}\left(1-\tau^\aug_J\right)^{-1} - \prod_{F\vee G\subset I}\left(1-\tau^\aug_I\right)^{-1}\right)\\
&=&  \prod_{F\vee G\subset I}\left(1-\tau^\aug_I\right)^{-2} \left(\prod_{F\subset H \subsetneq F\vee G}\left(1-\tau^\aug_H\right)^{-1} - 1\right)
\left(\prod_{G \subset J \subsetneq F\vee G}\left(1-\tau^\aug_J\right)^{-1} - 1\right),
\end{eqnarray*}
which vanishes because $H$ and $J$ are incomparable for any $H$ appearing in the second product and $J$ appearing in the third
product.

Next, we check that $\kappa_M$ vanishes on a generator of $\cJ_2^\aug$.  
Fix a flat $F$ of rank 1 and an element $e\in F$, so that
$$\kappa_M(u_F) = 1 - \prod_{F\subset H}\left(1-\tau^\aug_H\right)^{-1} = 1 - \prod_{e\in H}\left(1-\tau^\aug_H\right)^{-1} = 1 - \prod_{e\notin K}\left(1-\tau^\aug_K\right).$$
Then $$\kappa_M(u_F^2) = \left(1 - \prod_{e\in H}\left(1-\tau^\aug_H\right)^{-1} \right)\left(1 - \prod_{e\notin K}\left(1-\tau^\aug_K\right)\right),$$
which vanishes by Equation \eqref{eq:teq2}.  Now for any flat $G$, we have
\begin{eqnarray*}&&\kappa_M\Big(u_F(u_G-u_{F\vee G})\Big)\\ &=& 
\left(1 - \prod_{e\notin K}\left(1-\tau^\aug_K\right)\right)
\prod_{F\vee G\subset I}\left(1-\tau^\aug_I\right)^{-1}
\left(\prod_{G \subset J \subsetneq F\vee G}\left(1-\tau^\aug_J\right)^{-1} - 1\right),
\end{eqnarray*}
which also vanishes by Equation \eqref{eq:teq2} because $e\notin J$ for any $J$ appearing in the last product.

We have now proved that $\kappa_M$ descends to a homomorphism $\overline{\kappa}_M\colon A^\aug(M)\to K^\aug(M)$.  
M\"{o}bius inversion tells us that $$1-\tau_F^\aug = \prod_{F\subset G}\Big(1-\overline{\kappa}_M(h_G)\Big)^{-\mu(F,G)},$$
where $\mu$ is the M\"obius function on the lattice of flats of $M$.  This, along with the surjectivity statement of Theorem \ref{Kmatroid}(2),
implies that $\overline{\kappa}_M$ is surjective.
By Lemma \ref{tf} applied to the variety $X_{\Sigma_M^\aug}$, $A^\aug(M)$ and $K^\aug(M)$ are free abelian groups of the same rank, thus $\overline{\kappa}_M$
is an isomorphism.
We may then take $\zeta_M^\aug$ to be the inverse of $\overline{\kappa}_M$.
\end{proof}

\begin{proposition}\label{prop:commute}
Let $M$ be the matroid associated with a hyperplane arrangement $\mathcal{A}$. The following diagrams commute:
\begin{center}
\begin{tikzcd}
K(M) \arrow[r, "\zeta_M"] \arrow[d]
& A(M) \arrow[d] &&K^{\aug}(M) \arrow[r, "\zeta_M^{\aug}"] \arrow[d]
& A^{\aug}(M) \arrow[d] \\
K(\WA) \arrow[r, "\zeta_{\mathcal{A}}"]
& A(\WA)&&K(\WA^{\aug}) \arrow[r, "\zeta_{\mathcal{A}}^{\aug}"]
& A(\WA^{\aug}).
\end{tikzcd}
\end{center}
\end{proposition}

\begin{proof}
We do the non-augmented case.  Since $\zeta_\cA$ takes $\eta_F$ to $h_F$ and $\zeta_M$ is characterized in the statement of Theorem \ref{thm:exceptionalmatroid}, 
we need to show that
$$\eta_F = 1 - \prod_{F\subset G} (1-\tau_G)^{-1}.$$
By considering first Chern classes, we have $$\mathcal{L}_F \cong  \mathcal{L}_E \otimes \bigotimes_{F \subset G \subsetneq E} \mathcal{O}(-D_G)$$
as line bundles on the wonderful variety $\WA$.
Since $\tau_E = 1 - [\mathcal{L}_E]$, $\tau_G = 1 - [\mathcal{O}(-D_G)]$ for $G \not= E$, and $\eta_F =1 - [\mathcal{L}_F^{-1}]$, the result follows.
\end{proof}

We now discuss an analogue of the projection formula that will allow us to interpret the Euler characteristic maps on matroid $K$-rings using geometry.  Let $\cU_E$ denote the Boolean arrangement on the ground set $E$, consisting of the coordinate
hyperplanes in $\mathbb{F}^E$, realizing the Boolean matroid $U_E$.  The wonderful variety of $\mathcal{U}_E$ is a toric variety, called the \textbf{permutohedral variety}, which we denote $X_E$. The augmented wonderful variety is also a toric variety, called the \textbf{stellahedral variety}, which we denote $X_E^{\aug}$. For any matroid $M$, there is an open embedding $\iota \colon X_{\Sigma_M} \hookrightarrow X_{E}$, so we have restriction maps $\iota^* \colon K(X_{E}) \to K(M)$ and $\iota^* \colon A(X_E) \to A(M)$. Similarly, there is an open embedding $\iota^{\aug} \colon X_{\Sigma_M^{\aug}} \to X_E^{\aug}$ and corresponding restriction maps on $K$-rings and Chow rings. These maps are characterized by the property that
$$\iota^* t_S = \begin{cases} t_S, & S \text{ is a flat of }M \\ 0, & \text{otherwise} \end{cases} \and \iota^* \tau_S = \begin{cases} \tau_S, & S \text{ is a flat of }M \\ 0, & \text{otherwise,} \end{cases}$$
and similarly in the augmented setting.

\begin{lemma}\label{lem:resformula}
For any subset $S \subset E$, let $\bar{S}$ be the closure of $S$ in $M$. We have
\begin{equation*}
\begin{aligned}
\iota^*(\eta_S) &= \eta_{\bar S}\in K(M) & \quad \quad 
(\iota^\aug)^*(\eta^\aug_S) &= \eta^\aug_{\bar S}\in K^\aug(M)\\
\iota^*(h_S) &= h_{\bar S} \in A(M) & 
(\iota^\aug)^*(h^\aug_S) &= h^\aug_{\bar S}\in A^\aug(M).
\end{aligned}
\end{equation*}
\end{lemma}

\begin{proof}
We have that
$$\iota^* h_S = -\iota^* \sum_{S \subset S'} t_{S'} = -\sum_{\substack{S \subset F\\ \text{$F$ a flat}}} t_F = h_{\bar{S}}.$$
The other statements are similar. 
\end{proof}

There exist unique Chow classes
$$\Delta_M\in A(U_E)\and \Delta_M^\aug\in A^\aug(U_E),$$
called the {\bf Bergman class} and {\bf augmented Bergman class}, respectively, with the properties that,
for any $\xi\in A(U_E)$ and $\xi^\aug\in A^\aug(U_E)$, we have
$$\deg_M(\iota^*\xi) = \deg_{U_E}(\xi\cdot \Delta_M)\and \deg_M((\iota^\aug)^*\xi^\aug) = \deg_{U_E}(\xi^\aug\cdot \Delta_M^\aug).$$
In the non-augmented setting, this is proved in \cite[Theorem 4.2.1]{BES}. The same argument works in the augmented setting, using \cite[Theorem 1.11]{EurHuhLarson}.

\begin{proposition}
\label{prop:compatibility}
The following diagrams commute:
	\begin{center}
    \begin{tikzcd}
	    K(U_E) \ar[r, "\zeta_{U_E}"] \ar[d, "\iota^*"] & A(U_E) \ar[d, "\iota^*"] & & K^\aug(U_E) \ar[r, "\zeta_{U_E}^{\aug}"] \ar[d, "(\iota^\aug)^*"] & A^\aug(U_E) \ar[d, "(\iota^\aug)^*"]\\
	    K(M) \ar[r, "\zeta_{M}"] & A(M) & &  K^{\aug}(M) \ar[r, "\zeta^{\aug}_{M}"] & A^{\aug}(M).
    \end{tikzcd}
    \end{center}
Additionally, for any $\xi \in K(U_E)$ and $\xi^{\aug} \in K^\aug(U_E)$, we have
$$\chi\big(M, \iota^*(\xi)\big) = \chi\left(U_E, \xi \cdot \zeta_{U_E}^{-1}\left(\Delta_M\right)\right), \qquad \text{and}$$
$$\chi^{\aug}\big(M, (\iota^\aug)^*\left(\xi^{\aug}\right)\big) = \chi\left(U_E, \xi^{\aug} \cdot (\zeta^{\aug}_{U_E})^{-1}\left(\Delta_M^{\aug}\right)\right).$$
\end{proposition}

\begin{proof}[Proof of Proposition \ref{prop:compatibility}]
Commutativity of the diagrams follows from Lemma~\ref{lem:resformula}.
For the second statement, we do the non-augmented case. For any $\xi \in K(U_E)$, 
we have 
	\begin{align*}
	\chi\big(M, \iota^*(\xi)\big)
	  &= \deg_M \left(\frac{\zeta_M\circ \iota^*(\xi)}{1 - h_E} \right) \\
	  &= \deg_M \left(\frac{\iota^* \circ\zeta_{U_E}(\xi)}{1 - h_E} \right) \\
      &= \deg_{U_E} \left( \frac{\zeta_{U_E}(\xi) \cdot \Delta_M}{1 - h_E}\right) \\
      &= \chi\left(U_E, \xi \cdot \zeta^{-1}_{U_E}(\Delta_M)\right). 
	\end{align*}
The augmented case is similar.
\end{proof}

As a corollary, we are to prove the following characterizing properties of $\zeta_{U_E}$ and $\zeta_{U_E}^\aug$, which demonstrate that they agree with the
maps denoted $\zeta$ in \cite[Corollary 10.6]{BEST} and \cite[Theorem 8b]{EurHuhLarson}. 

\begin{corollary}\label{cor:structuresheaf}
Let $\mathcal{A}$ be a hyperplane arrangement equipped with a choice of linear function cutting out each hyperplane, which induces embeddings $\WA\subset X_E$
and $\WA^{\aug}\subset X_E^\aug$. 
Then $$\zeta_{U_E}\big([\cO_{\WA}]\big) = [\WA]\and \zeta_{U_E}^\aug\big([\cO_{\Waug}]\big) = [\Waug].$$
\end{corollary}

\begin{proof}
Let $M$ be the matroid associated with $\cA$.
For any $\xi \in K(X_E) = K(U_E)$, we have
$$\chi\left(X_E, \xi \cdot \zeta_{U_E}^{-1}(\Delta_M)\right)
= \chi\left(U_E, \xi \cdot \zeta_{U_E}^{-1}(\Delta_M)\right) = \chi(M, \iota^*(\xi)) = \chi(\WA, \iota^*(\xi)) 
= \chi(X_E, \xi \cdot [\cO_{\WA}]),$$
where the second equality comes from Proposition \ref{prop:compatibility} and the fourth comes from the projection formula
on the permutohedral variety.
The nondegeneracy of the pairing $(x, y) \mapsto \chi(X_E, xy)$ \cite[Theorem 1.3]{anderson2015operational} 
implies that $\zeta_{U_E}^{-1}(\Delta_M) = [\cO_{\WA}]$. 
Applying $\zeta_{U_E}$, we find that $$[\WA] = \Delta_M = \zeta_{U_E}\big([\cO_{\WA}]\big).$$  The augmented case is similar.
\end{proof}

Motivated by Corollary \ref{cor:structuresheaf}, we define
$$[\cO_M] \coloneqq \zeta_{U_E}^{-1}(\Delta_M) \in K(U_E)\and [\cO_M^{\aug}] \coloneqq (\zeta_{U_E}^{\aug})^{-1}(\Delta_M^{\aug}) \in K^\aug(U_E)$$ for any loopless matroid $M$. 
Proposition~\ref{prop:compatibility} may then be interpreted as a purely combinatorial projection formula. 

We conclude this section with the following lemma,
which we will need in Section \ref{sec:FY Snap}. 

\begin{lemma}\label{exceptional FY}
The exceptional isomorphism $\zeta_M\colon K(M)\to A(M)$ has the following properties:
\begin{itemize}
\item $\zeta_M(\tau_\varnothing) = t_\varnothing$
\item $\zeta_M(\tau_E) = \frac{t_E}{1+t_E}$
\item For any flat $F$, $\zeta_M(\tau_F) = t_F + \text{terms of degree at least 2}$.
\end{itemize}
\end{lemma}

\begin{proof}
By Proposition \ref{prop:compatibility}, it is sufficient to prove these statements
when $M = U_E$ is the Boolean matroid.
The first two statements for the Boolean matroid are proved in the course of the proof of
\cite[Theorem 10.11]{BEST}.  

For the last statement, let $F$ be a nonempty flat.
After restricting from $X_{\Sigma^\aug_M}$ to $X_{\Sigma_M}$, we have that
$$\tau_F = -\sum_{F\subset G} \mu(F,G) \eta_G + \text{terms of degree at least 2 in $\{\eta_G \mid F \subset G\}$}.$$
Applying $\zeta_M$, we find that \begin{eqnarray*}\zeta_M(\tau_F) &=& - \sum_{F\subset G}\mu(F,G) h_G + \text{terms of degree at least 2 in $\{h_G \mid F \subset G\}$}\\
&=& t_F + \text{terms of degree at least 2},\end{eqnarray*}
which completes the proof.
\end{proof}

\section{Adams operations and Serre duality}\label{ssec:structure}
We now discuss some additional structures on matroid $K$-rings. 
The fact that $K(M)$ 
and $K^{\aug}(M)$ 
are (by definition) $K$-rings of toric varieties $X_{\Sigma_M}$ and $X_{\Sigma^{\aug}_M}$ endows them with the structure of {\bf augmented \boldmath{$\la$}-rings} \cite[Expos\'e V, Exemple 3.9.1]{SGA6}. This means that we have a rank function $\epsilon$ that takes values in $\Z$, and for each natural
number $k$, we have operations $\la^k$ and $\Psi^k$ (the latter called {\bf Adams operations})
characterized by the property that
$$\lambda^k([\mathcal{E}]) = [\wedge^k \mathcal{E}]\and \Psi^k([\mathcal{L}]) = [\otimes^k \mathcal{L}]$$
for any vector bundle $\mathcal{E}$ and any line bundle $\mathcal{L}$. 
Since our simplicial generators $\eta_F$ are all of the form $1-[\mathcal{L}]$ for some line bundle $\mathcal{L}$, we have
$\epsilon(\eta_F)=0$, and
$$\Psi^k(\eta_F) = \sum_{i=1}^k (-1)^{i+1}\binom{k}{i} \eta_F^i.$$
The formula for augmented simplicial generators is identical.
Note that the Adams operations are ring homomorphisms, which is not at all combinatorially obvious from the above formula.
The Adams operations become simultaneously diagonalizable after tensoring with $\mathbb{Q}$, and their eigenspaces (which are independent of $k>1$)
map isomorphically to the graded pieces of the Chow ring via 
the Chern character.
We also have a \textbf{duality automorphism} $D$, characterized by the property that
$D([\mathcal{E}]) = [\mathcal{E}^\vee]$ for any vector bundle $\mathcal{E}$.
In terms of the simplicial generators, this takes the form
$$D(\eta_F) = \frac{-\eta_F}{1-\eta_F} = -\eta_F - \eta_F^2 - \dotsb,$$ and similarly in the augmented setting.

We note that the operations $\epsilon$, $\lambda^k$, $\Psi^k$, and $D$ all commute with 
the maps $\iota^*$ and $(\iota^\aug)^*$ introduced in Section \ref{sec:matroidk} because they are compatible with functorial maps between $K$-rings of varieties. 

On a $d$-dimensional smooth projective variety $X$ with canonical bundle $\omega_X$, Serre duality implies that, for any vector bundle $\mathcal{E}$, 
$$\chi(X, \mathcal{E}) = (-1)^d \chi(X, \omega_X \otimes \mathcal{E}^{\vee}).$$
In particular, this holds for $X = \WA$ or $\Waug$.  We will show that a similar formula holds on $K(M)$ for any matroid $M$,
even if $M$ is not realizable.  Our first task is to define classes $\omega_M \in K(M)$ and $\omega_M^{\aug} \in K^{\aug}(M)$
that will play the roles of the canonical bundles.  

For any matroid on the ground set $E$, the {\bf matroid polytope} $P(M)\subset \R^E$ is defined to be the convex hull
of the indicator functions of bases of $M$, and the {\bf independence polytope} $\IP(M)\subset\R^E$ is defined
to be the convex hull of indicator functions of independent subsets of $M$.

Recall that, on a smooth projective toric variety with fan $\Sigma$, there is a correspondence between torus equivariant
nef line bundles and lattice polytopes whose normal fans coarsen $\Sigma$. For any matroid on the ground set $E$,
$P(M)$ coarsens $\Sigma_{U_E}$ and $\IP(M)$ coarsens $\Sigma_{U_E}^\aug$, so we obtain line bundles
$[P(M)]$ and $[\IP(M)]$ on $X_E$ and $X_E^{\aug}$, respectively.  

Let $M^\perp$ be the matroid dual to $M$,
characterized by the property that the bases of $M^\perp$ are the complements of the bases of $M$.
Using the standard description of the canonical bundle of a smooth proper toric variety \cite[Section~4.3]{Ful93}, one can check that $c_1(\omega_{X_E}) = x_{\emptyset} + x_E = -\sum_{F \not= \emptyset, E} x_F$ and $c_1(\omega_{X_E}^{\aug}) = x_E - \sum_{e \in E} y_e = -\sum_{F \not= E} x_F - \sum_{e \in E} y_e$. 
We define the classes
$$\omega_M \coloneqq \iota^*\left(\omega_{X_E} \cdot [P(M^\perp)]\right) 
\in K(M)$$
and 
$$\omega_M^\aug \coloneqq (\iota^\aug)^*\left(\omega_{X^{\aug}_E} \cdot [\IP(M^\perp)]\right) 
\in K^\aug(M).$$

These definitions are motivated by the following proposition.

\begin{proposition}\label{canonical}
Let $\cA$ be a hyperplane arrangement with associated matroid $M$.
The isomorphisms $K(M)\cong K(\WA)$ and 
$K^\aug(M)\cong K(\Waug)$ take $\omega_M$ and $\omega_M^\aug$
to the canonical bundles of $\WA$ and $\Waug$, respectively.
\end{proposition}

\begin{proof}
The adjunction formula states that the canonical class of $\Waug$
is equal to the restriction of the canonical class of $X_E^{\aug}$ tensored with the determinant of the 
normal bundle to $\Waug$ inside of $X_E^{\aug}$.
The determinant of the normal bundle is equal to the restriction of $[\IP(M^\perp)]$ 
by \cite[Proposition 4.6 and Corollary 5.4]{EurHuhLarson}. The non-augmented case can be prove similarly, using  \cite[Theorem 7.10]{BEST}.
\end{proof}

The following theorem is a combinatorial version of Serre duality for $K$-rings of matroids.

\begin{theorem}\label{thm:SD}
For any $\xi \in K(M)$ and $\xi^{\aug} \in K^{\aug}(M)$, we have
$$\chi(M, \xi) = (-1)^{\rk M-1} \chi(M, \omega_M \cdot D(\xi)) \and \chi^{\aug}(M, \xi^{\aug}) = (-1)^{\rk M} \chi^{\aug}(M, D(\xi^{\aug})).$$
\end{theorem}

In the realizable case, Theorem \ref{thm:SD} follows from Proposition \ref{canonical} and Serre duality for the varieties
$\WA$ and $\Waug$. Our strategy is to deduce the general case using the concept of valuativity. 
Given an abelian group $A$, a function  $f \colon \{\text{matroids on } E\} \to A$ is called \textbf{valuative}
if it factors through the map that takes $M$
to the indicator function of its matroid polytope $P(M)\subset\R^E$.
That is, for any matroids $M_1, \ldots, M_k$ and integers $a_1, \ldots, a_k$ such that $\sum a_i 1_{P(M_i)} = 0$, we
require that $\sum a_i f(M_i) = 0$.

\begin{remark}The notion of valuativity is typically defined on functions that take values
on the set of all matroids on $E$, whereas we are only considering loopless matroids in this paper.
This is not an important distinction, as the valuative group of loopless matroids on $E$ is a direct summand of the valuative group
of all matroids on $E$.  That is, a function on the set of loopless matroids on $E$ is valuative if and only if it extends to a valuative function on the set of all matroids on $E$, if and only if it extends by zero to a valuative function on the set of all matroids on $E$.
\end{remark}

\begin{lemma}\label{everything is valuative}
Fix a pair of classes $\xi\in K(U_E)$ and $\xi^\aug\in K^\aug(U_E)$.  The following four $\Z$-valued functions are valuative:
\begin{align*}
M &\mapsto \chi(M, \iota^*\xi)\\
M &\mapsto \chi^\aug(M, (\iota^\aug)^*\xi^\aug)\\
M &\mapsto (-1)^{\rk M - 1}\chi\big(M, \omega_M\cdot D(\iota^*\xi)\big)\\
M &\mapsto (-1)^{\rk M}\chi^\aug\big(M, \omega_M^\aug\cdot D((\iota^\aug)^*\xi^\aug)\big).
\end{align*}
\end{lemma}

\begin{proof}
We begin with the first function.
By Proposition~\ref{prop:compatibility}, $\chi(M, \iota^*\xi) = \chi(X_E, [\cO_M] \cdot \xi)$. Recall that $[\mathcal{O}_M] = \zeta_{U_{E}}^{-1}(\Delta_M)$, and the function $M \mapsto \Delta_M$ is valuative by \cite[Corollary 7.11]{BEST}. Therefore the function $M \mapsto [\mathcal{O}_M]$ is valuative, and the result follows by the linearity of the Euler characteristic
map on $K(X_E)$. 
The proof of valuativity of the second function is similar, except that we now use \cite[Proposition 4.7]{EurHuhLarson} 
to establish the valuativity of the function $M\mapsto [\mathcal{O}_M^{\aug}]$. 

After applying $\zeta_{U_E}^\aug$, it follows from \cite[Proposition 4.7 and Corollary 6.5]{EurHuhLarson} that the function
$M\mapsto [\cO_M^\aug]\cdot [\IP(M^\perp)]$ is valuative. This implies the valuativity of the fourth function. One can argue similarly for the third function. 
\end{proof}

\begin{proof}[Proof of Theorem~\ref{thm:SD}]
We have already used Proposition \ref{canonical} to prove the theorem for realizable matroids.
We have shown that each side of both claimed equalities is valuative.
The full theorem now follows from a result of Derksen and Fink \cite{DF10}
(see also \cite[Lemma 5.9]{BEST}), which says that 
the matroid polytope of any matroid can be expressed as a linear combination of indicator functions of matroid polytopes of realizable matroids.
\end{proof}

\begin{question}
Can Theorem~\ref{thm:SD} be applied to give interesting identities for matroids? For example, let $f(\ell) \coloneqq \chi(M, \omega_M^{\ell})$; then Theorem~\ref{thm:SD} implies that $f(\ell) = (-1)^{r-1} f(1 - \ell)$. Does this statement admit an elementary proof,
without using valuativity and Serre duality for wonderful varieties?  
\end{question}

\section{Euler characteristics and simplicial Snapper polynomials} \label{sec:applications}
In this section, we give purely combinatorial formulas for the Euler characteristic of a monomial in the simplicial generators
on $K(M)$ or $K^\aug(M)$.  
Given a pair of tuples $$\textbf{m} = (m_F \mid F \text{ a nonempty flat}) \and \textbf{m}' = (m'_F \mid F \text{ a nonempty flat}),$$
we say that $\textbf{m}'\leq \textbf{m}$ if $m'_F\leq m_F$ for all $F$.
We say that $\textbf{m}$ satisfies the {\bf Hall--Rado condition} if, for all $\textbf{m}'\leq \textbf{m}$, 
we have $$\operatorname{rank}\bigcup_{m'_F>0} F \;\geq\; \sum_{F \not= \emptyset} m'_F.$$
If the inequality is always strict, we say that $\textbf{m}$ satisfies the {\bf dragon Hall--Rado condition}.
The following theorem computes the degrees of the monomials $h^{\textbf{m}}$ and $(h^\aug)^{\textbf{m}}$.

\begin{theorem}\label{degreeChow}
If $\sum m_F = \rk M - 1$, then  \cite[Theorem 5.2.4]{BES}
$$\deg_M(h^{\textbf{m}}) = \begin{cases} 1 & \text{if $\textbf{m}$ satisfies the dragon Hall--Rado condition}\\ 0 & \text{otherwise.}\end{cases}$$ 
If $\sum m_F = \rk M$, then  \cite[Theorem 1.3]{EHLPolymatroid}
$$\deg^\aug_M((h^\aug)^{\textbf{m}}) = \begin{cases} 1 & \text{if $\textbf{m}$ satisfies the Hall--Rado condition}\\ 0 & \text{otherwise.}\end{cases}$$
\end{theorem}

The $K$-theoretic analogue of Theorem \ref{degreeChow} is almost exactly the same, except without
the condition on $\sum m_F$.

\begin{theorem}
\label{thm:euler-simplicial-matroid}
For any $m$, we have
$$\chi(M, \eta^{\textbf{m}}) = \begin{cases} 1 & \text{if $\textbf{m}$ satisfies the dragon Hall--Rado condition}\\ 0 & \text{otherwise}\end{cases}$$ and 
$$\chi^\aug(M, (\eta^\aug)^{\textbf{m}}) = \begin{cases} 1 & \text{if $\textbf{m}$ satisfies the Hall--Rado condition}\\ 0 & \text{otherwise.}\end{cases}$$
\end{theorem}

\begin{proof}
We prove the first statement; the augmented case is similar.  
If $\sum m_F \geq \rk M$, then $\eta^{\textbf{m}} = 0$ and $\textbf{m}$ fails to satisfy the dragon Hall--Rado condition.  If $\sum m_F < \rk M$,
we let $$k = \rk M - 1 - \sum m_F \geq 0,$$ and
define $\tilde{\textbf{m}}$ by putting $\tilde m_E = m_E + k$
and $\tilde m_F = m_F$ for all $F\neq E$.  Then $$\chi(M, \eta^{\textbf{m}}) 
=  \deg_M\left(\frac{\zeta_M(\eta^{\textbf{m}})}{1-h_E}\right) 
= \deg_M\left(\frac{h^{\textbf{m}}}{1-h_E}\right) 
= \deg_M\left(h^{\tilde{\textbf{m}}}\right).$$
We observe that $\textbf{m}$ satisfies dragon Hall--Rado if and only if $\tilde{\textbf{m}}$ does, thus our formula follows from Theorem \ref{degreeChow}.
\end{proof}

Given a proper variety $X$ and a finite tuple of line bundles $\cL = (\cL_1,\ldots,\cL_k)$ on $X$,
the function from $\Z^k$ to $\Z$ given by the formula
$$\Snap_\cL(\textbf{a}) \coloneqq \chi\left(X, \cL_1^{a_1}\otimes\cdots\otimes\cL_k^{a_k}\right)$$
is a polynomial of degree at most $\dim X$ \cite[Theorem 9.1]{MR0109156},
which we call the {\bf Snapper polynomial}.
Given an integer $x$ and a natural number $d$, let $$x^{(d)} = \frac{x(x+1)\cdots(x+d-1)}{d!} = \binom{d+x-1}{d} = (-1)^d\binom{-x}{d}.$$ 
If $\textbf{a}\in\Z^k$ and $\textbf{d}\in\N^k$, let $$\textbf{a}^{(\textbf{d})} \coloneqq a_1^{(d_1)}\cdots a_k^{(d_k)}.$$

\begin{lemma}\label{oh snap}
If $\sigma_i = 1 - [\cL_i^{-1}]$ for each $i\in\{1,\ldots,k\}$,
then $$\Snap_\cL(\textbf{a}) = \sum_{\textbf{d}\in \N^k}\chi\left(X, \sigma_1^{d_1} \cdots \sigma_k^{d_{k}}\right)\, \textbf{a}^{(\textbf{d})}.$$
\end{lemma}

\begin{proof}
We have 
$$[\cL_i^{a_i}] = \frac{1}{(1-\sigma_i)^{a_i}} = \sum_{d_i=0}^\infty a_i^{(d_i)} \sigma_i^{d_i}.$$
The lemma follows.
\end{proof}

Let $M$ be a matroid, and consider a tuple $\textbf{a} = (a_F \mid \text{$F$ a nonempty flat})$.
Let $$\Snap_M(\textbf{a}) \coloneqq \sum_m\chi(M, \eta^{\textbf{m}}) \textbf{a}^{(\textbf{m})} \and \Snap_M^\aug(\textbf{a}) \coloneqq \sum_m\chi^\aug(M, (\eta^\aug)^{\textbf{m}}) \textbf{a}^{(\textbf{m})}.$$
This definition is motivated by the following observation.

\begin{lemma}\label{snap is geometric}
Let $\cA$ be a hyperplane arrangement with associated matroid $M$.
Then $\Snap_M$ and $\Snap_M^\aug$
coincide with the geometrically defined Snapper polynomials for the varieties $\WA$ and $\Waug$
with respect to the line bundles $(\cL_F)$ and $(\cL_F^\aug)$.
\end{lemma}

\begin{proof}
By definition, $\eta_F = 1 - [\cL_F^{-1}]$.
The fact that $\Snap_M$ coincides with the Snapper polynomial for $\WA$
with respect to $(\cL_F)$ then follows from Lemma \ref{oh snap}.  The augmented case is identical.
\end{proof}

We can now rephrase Theorem \ref{thm:euler-simplicial-matroid} as a statement about Snapper polynomials.

\begin{corollary}
\label{cor:snapper-simplicial-matroid}
We have
$$\Snap_M(\textbf{a})\;\; = \sum_{\substack{\text{$\textbf{m}$ satisfies}\\ \text{dragon Hall--Rado}}} \textbf{a}^{(\textbf{m})}
\and
\Snap_M^\aug(\textbf{a})\;\; = \sum_{\substack{\text{$\textbf{m}$ satisfies}\\ \text{Hall--Rado}}} \textbf{a}^{(\textbf{m})}.$$
\end{corollary}

\begin{remark}
In the special case where $M = U_E$ is Boolean, the first equality in \cref{cor:snapper-simplicial-matroid} appears in \cite[Theorem 11.3]{Postnikov}.
\end{remark}

\section{Snapper polynomials in the Feichtner--Yuzvinsky generators}\label{sec:FY Snap}
Section \ref{sec:applications} was about the Snapper polynomial of the wonderful variety of an arrangement 
with respect to the line bundles $\{\cL_F\mid\text{$F$ a nonempty flat}\}$, or of an arbitrary matroid with respect
to the corresponding $K$-classes.  This was fundamentally a ``simplicial'' construction, since the simplicial generators
for $K$-theory were defined in terms of these line bundles.  In this section, we consider the Snapper polynomial
with respect to a collection of line bundles related to the Feichtner--Yuzvinsky generators.  For simplicity, we work
only in the non-augmented setting.

Recall that, for all flats $F$, we have defined $$t_F \coloneqq \PhFY(x_F)\in A(M) \and \tau_F \coloneqq \PsFY(x_F)\in K(M).$$
For any tuple of natural numbers $\textbf{m} = (m_F\mid \text{$F$ a flat})$, let $$\mathbf{\tau}^{\textbf{m}} \coloneqq \prod_F \tau_F^{m_F}.$$
For any tuple of integers $\textbf{a} = (a_F \mid \text{$F$ a flat})$, we define
$$\Snap_M^{\FY}(\textbf{a}) \coloneqq \sum_{\textbf{m}} \chi(M, \mathbf{\tau}^{\textbf{m}}) \textbf{a}^{(\textbf{m})}.$$

If $M$ is the matroid associated with a hyperplane arrangement $\cA$,
let $D_{\emptyset}$ and  $D_E$ be the divisor classes on $\WA$ with first Chern classes $t_{\emptyset}$ and $t_E$ respectively, so that we have $\tau_F = 1 - [\mathcal{O}(-D_F)]$ for all flats $F$.
Lemma~\ref{oh snap} immediately implies that $\Snap_M^{\FY}$ is the Snapper polynomial of $\WA$ with respect to the tuple of line bundles $\left(\mathcal{O}(D_F)\mid \text{$F$ a flat}\right)$.

We now provide an explicit formula for the polynomial $\Snap_M^{\FY}(\textbf{a})$.
For any natural number $k$, let $\Flag_k$ be the collection
of flags of flats of the form
\[\calF = \{\varnothing = F_0 \subsetneq F_1 \subsetneq \cdots \subsetneq F_k = E\}.\]
If $F$ and $G$ are incomparable flats, then $\tau_F \tau_G = 0\in K(M)$, therefore $\Snap_M^{\FY}(\textbf{a})$ does not contain any monomials that are multiples of $a_F a_G$.
This implies that we may write
$$\Snap_M^{\FY}(\textbf{a}) = \sum_{k\geq 1}\sum_{\cF\in\Flag_k}\sum_{\textbf{m}} c(\cF,\textbf{m})\, a_{F_0}^{(m_0)}a_{F_1}^{(m_1+1)}a_{F_2}^{(m_2+1)}\cdots a_{F_{k-1}}^{(m_{k-1}+1)}a_{F_k}^{(m_k)},$$
where $\textbf{m}$ ranges over all tuples $(m_0,\ldots,m_k)$ of natural numbers.
It remains only to compute the constants $c(\cF,\textbf{m})$.

\begin{remark}
It may be slightly confusing that the first and last exponents are $(m_0)$ and $(m_k)$, whereas the middle exponents
are $(m_i+1)$ for $1 \le i < k$.  The point is that a particular ``monomial'' appears in the summand indexed by $\cF$ if and only if
the nonempty proper flats in its support are precisely $\{F_1,\ldots,F_{k-1}\}$.
Also, this convention leads to a tidier formula for $c(\cF, \textbf{m})$.
\end{remark}

Given $\cF \in \Flag_k$ and $i\in\{1,\ldots,k\}$, consider the matroid
$M^{F_{i}}_{F_{i-1}}$ obtained by contracting $F_{i-1}$ and deleting the complement of $F_{i}$, 
whose poset of flats is canonically isomorphic
to the interval $[F_{i-1},F_{i}]$ in the poset of flats of $M$.  Let $d_i = \rk M^{F_{i}}_{F_{i-1}} - 1$.
Let $\mu^j(M)$ denote the absolute value of the coefficient of $t^{\rk M - 1 - j}$ in the reduced characteristic polynomial of $M$.

\begin{theorem} 
\label{thm:snapper-FY-wonderful}
For any $k\geq 1$, $\cF \in \Flag_k$, and $\textbf{m} = (m_0,\ldots,m_k)$, we have
$$c(\cF,\textbf{m}) = \sum_{\substack{\textbf{e},\textbf{f}\in\mathbb{N}^k\\ e_1 = m_0}}(-1)^{|\textbf{e}|+|\textbf{f}|}\prod_{i=1}^k \binom{d_i-e_i}{f_i}\mu^{e_i}(M^{F_{i}}_{F_{i-1}}) \binom{m_i}{e_{i+1}+f_i-m_i,m_i-f_i,m_i-e_{i+1}},$$
where we adopt the convention that $e_{k+1}= 0$.
\end{theorem}

\begin{corollary} 
\label{cor:degrees}
Fix an element $\cF \in \Flag_k$ and a tuple of natural numbers $\textbf{m} = (m_0,\ldots,m_k)$ with $\sum m_i = \rk M - k$.  Then the degree
$$\deg_M\left(t_{F_0}^{m_0}t_{F_1}^{m_1+1}\cdots t_{F_{k-1}}^{m_{k-1}+1}t_{F_k}^{m_k}\right)$$
is equal to 
$$(-1)^{\rk M - k}\sum_{\substack{\textbf{e}\in\mathbb{N}^k\\ e_1 = m_0}}\prod_{i=1}^k \mu^{e_i}(M^{F_{i}}_{F_{i-1}}) \binom{m_i}{e_{i+1}-e_i+d_i-m_i,m_i-d_i+e_i,m_i-e_{i+1}},$$
again with the convention that $e_{k+1}=0$.
\end{corollary}

\begin{proof}
By the definition of $\chi$ and the third item of Lemma \ref{exceptional FY}, we have
$$\deg_M\left(t_{F_0}^{m_0}t_{F_1}^{m_1+1}\cdots t_{F_{k-1}}^{m_{k-1}+1}t_{F_k}^{m_k}\right) = \chi\left(M,\tau_{\varnothing}^{m_0}\tau_{F_1}^{m_1+1}\cdots \tau_{F_{k-1}}^{m_{k-1}+1}\tau_{F_k}^{m_k}\right)
= c(\cF,\textbf{m}).$$ 
In the formula for $c(\cF,\textbf{m})$ in Theorem \ref{thm:snapper-FY-wonderful}, we see that a particular term vanishes unless $e_i + f_i\leq d_i$
for all $i\in\{1,\ldots,k\}$.  Taking the sum over all $i$, this implies that
$$|e| + |f| \leq |d| = \rk M - k = m_0+\cdots+m_k.$$
On the other hand, a term also vanishes unless $m_i\leq e_{i+1} + f_i$ for all $i\in\{1,\ldots,k\}$.  Taking the sum over all $i$ and including the equality $e_0 = m_0$,
this implies that $$m_0+\cdots+m_k \leq |e| + |f|.$$
Thus the aforementioned inequalities are all equalities, and we can simplify our formula for $c(\cF,\textbf{m})$ by setting $f_i = d_i - e_i$ for all $i$.
The result follows.
\end{proof}

If we take $m_0 = 0 = m_k$ in Corollary \ref{cor:degrees}, we recover a result of Eur \cite[Theorem 3.2]{Eur20}.

\begin{corollary}
\label{cor:volum-FY}
Fix an element $\cF \in \Flag_k$ and a tuple of natural numbers $(m_1,\ldots,m_{k-1})$ with $(m_1+1) + \cdots + (m_{k-1}+1) = \rk M - 1$.
For all $i\in\{1,\ldots,k\}$, let $e_i = (m_1-d_1)+\cdots+(m_{i-1}-d_{i-1})$, and let $e_{k+1}=0$.
Then we have $$\deg_M\left(t_{F_1}^{m_1+1}\cdots t_{F_{k-1}}^{m_{k-1}+1}\right) = (-1)^{\rk M - k}\prod_{i=1}^k \mu^{e_i}(M^{F_{i}}_{F_{i-1}}) \binom{m_i}{e_{i+1}}.$$
\end{corollary}

\begin{proof}
The fact that $m_0 = 0 = m_k$ implies that the only nonzero term in Corollary \ref{cor:degrees} is the one in which $e_i = (m_1-d_1)+\cdots+(m_{i-1}-d_{i-1})$ for all $i$.
\end{proof}

The remainder of this section is devoted to the proof of Theorem \ref{thm:snapper-FY-wonderful}.
We begin by analyzing a particular specialization of $\Snap_M^{\FY}$, where we set all variables except $a_{\emptyset}$ and $a_E$ to $0$. 

\begin{lemma}\label{ups}
We have
$$\Snap_M^{\FY}(a_{\emptyset}, \mathbf{0}, a_E) = \sum_{e,f\in\mathbb{N}} (-1)^{e + f}\binom{\rk M - 1 - e}{f} \mu^e(M) a_{\emptyset}^{(e)} a_E^{(f)}.$$
\end{lemma}

\begin{proof}
By definition, we have
$$\Snap_M^{\FY}(a_{\emptyset}, \mathbf{0}, a_E) = \sum_{e,f\in\mathbb{N}}\chi(M, \tau_\varnothing^e \tau_E^f) a_{\emptyset}^{(e)} a_E^{(f)}.$$
We also have 
$$\chi(M, \tau_{\varnothing}^e \tau_{E}^{f}) =  \deg_M\left(\frac{\zeta_{M}(\tau_{\varnothing}^{e} \tau_{E}^{f})}{1 - h_E}\right)
= \deg_M\left(\frac{\zeta_{M}(\tau_{\varnothing}^{e} \tau_{E}^{f})}{1 + t_E}\right).$$
By Lemma \ref{exceptional FY}, this is equal to 
$$\deg_M\left(\frac{t_\varnothing^e t_E^f}{(1+t_E)^{f+1}}\right) = \sum_\ell (-1)^{\ell-f}\binom{\ell}{f}\deg_M(t_\varnothing^e t_E^\ell).$$
The degree of $t_\varnothing^e t_E^\ell$ vanishes unless $e + \ell = \rk M -1$, in which case it is equal to $(-1)^{\rk M - 1}\mu^e(M)$ \cite[Proposition 9.5]{AHK},
thus
$$\chi(M, \tau_{\varnothing}^e \tau_{E}^{f})  = (-1)^{e + f}\binom{\rk M - 1 - e}{f} \mu^e(M).$$
This completes the proof.
\end{proof}

For the statement and proof of the next proposition, 
it will be convenient to regard $\Snap_M^{\FY}(\textbf{a})$ as a function that takes inputs $a_S$ for all subsets $S \subset E$, with the property
that the coefficient of any monomial involving $a_S$ is zero if $S$ is not a flat of $M$.
It will also be convenient to allow $M$ to have loops, with the convention that $\Snap_M^{\FY}(\textbf{a}) = 0$ whenever $M$ has a loop.
This allows us to define the matroid $M_F^G$ on the ground set $G\setminus F$ with respect to an arbitrary pair of sets $F\subset G$; note that
this matroid is loopless if and only if $F$ is a flat.

Fix a subset $G\subset E$, and for any $\textbf{a} = (a_F)$, define $\textbf{a}'$ by putting $a'_G = a_G - 1$ and $a'_F = a_F$ for all $F\neq G$.
For any polynomial $P$ in $\textbf{a}$, define $\partial_G P$ by putting
$$\partial_GP(\textbf{a}) \coloneqq P(\textbf{a}) - P(\textbf{a}').$$
Note that if $P$ does not depend on $a_G$, then $\partial_G P = 0$. 
For any pair of subsets $F\subset G$, let $\textbf{a}|_{[F,G]}$ be the restriction of $\textbf{a}$ to this interval, which we identify with the collection of subsets of $G\setminus F$.

\begin{proposition}\label{the recursion}
For any proper nonempty subset $G\subset E$, we have
$$\partial_G\Snap_M^{\FY}(\textbf{a}) = \Snap^{\FY}_{M_\varnothing^G}(\textbf{a}|_{[\varnothing,G]})\,\Snap^{\FY}_{M_G^E}(\textbf{a}|_{[G,E]}).$$
\end{proposition}

\begin{proof}
We will first prove the statement when $M$ is the matroid associated with a hyperplane arrangement $\cA$, and then deduce the general case from valuativity.
Note that if $G$ is not a flat of $M$, then $M_{\emptyset}^G$ has a loop and both sides vanish.  We therefore assume that $G$ is a flat.

Consider the hyperplane arrangements $$\cA_{\varnothing}^G = \{H_e/L_G \mid e\in G\} \and \cA_G^E \coloneqq \{H_e \cap L_G\mid e\notin G\}.$$
We observe that $\cA_\varnothing^G$ is an arrangement in the vector space $L^G$ with associated matroid $M_\varnothing^G$, and $\cA_G^E$
is an arrangement in the vector space $L_G$ with associated matroid $M_G^E$.

There is a short exact sequence
\[
0 \to \calO\left((t_G - 1) D_G + \sum_{F \ne G} t_F D_F\right) \to \calO\left(\sum t_F D_F\right) \to \calO_{D_G}\left(\sum t_F D_F\right) \to 0. 
\]
Taking Euler characteristics, we find that
$$\partial_G\Snap_M^{\FY}(\textbf{a}) = \Snap_M^{\FY}(\textbf{a}) - \Snap_M^{\FY}(\textbf{a}')
= \chi\left(D_G, \calO_{D_G}\left(\sum t_F D_F\right) \right).$$
By \cite[Page 482]{dCP95}, the divisor $D_G$ is isomorphic to $W_{\!\calA_\varnothing^G} \times W_{\!\calA_G^E}$.
Furthermore, the restriction of $\calO_{D_G}(\sum t_F D_F)$ to $D_G$ is isomorphic to the pullback from $W_{\!\calA_\varnothing^G}$ of $\calO_{D_G}(\sum_{F \le G} t_F D_F)$
tensored with the  pullback from $W_{\!\calA_G^E}$ of $\calO_{D_G}(\sum_{G \le F} t_F D_F)$ \cite[Proposition 2.20]{BHMPW20a}.
Then by the K\"{u}nneth formula, we have
\begin{equation*}\begin{split}
\chi\left(D_G, \calO_{D_G}\left(\sum t_F D_F\right)\right) &= \chi\left(W_{\!\calA_\varnothing^G}, \calO_{D_G}\left (\sum_{F \le G} t_F D_F \right )\right) \cdot \chi\left(W_{\!\calA_G^E},  \calO_{D_G}\left (\sum_{G \le F} t_F D_F \right)\right) \\ 
&= \Snap^{\FY}_{M_\varnothing^G}(\textbf{a}|_{[\varnothing,G]})\,\Snap^{\FY}_{M_G^E}(\textbf{a}|_{[G,E]}).
\end{split}\end{equation*}

This completes the proof in the realizable case.
For the general case, it will suffice to show that both sides of the equation in the statement of the proposition are valuative invariants of $M$. It follows from Lemma~\ref{everything is valuative} that $\Snap_M^{\FY}(\textbf{a})$ is valutive, and therefore that
$\partial_G \Snap_M^{\FY}(\textbf{a})$ is valuative. The valuativity of $\Snap^{\FY}_{M_\varnothing^G}(\textbf{a}|_{[\varnothing,G]})\,\Snap^{\FY}_{M_G^E}(\textbf{a}|_{[G,E]})$ follows from the valuativity of $\Snap_M^{\FY}(\textbf{a})$ and general properties of valuativity \cite[Theorem 4.6]{McMullen2009} (see also \cite[Theorem A]{AS22}).
\end{proof}

\begin{remark}\label{that's why}
Our reason for regarding $\Snap_M^{\FY}(\textbf{a})$ as a polynomial with variables indexed by arbitrary subsets rather than flats, and for allowing matroids
with loops, is that it would not otherwise make sense to assert that the two sides of the equation in the statement of Proposition \ref{the recursion}
are valuative invariants of the matroid $M$.  When considering all matroids at once, we cannot know in advance whether or not $G$ is a flat.
\end{remark}

Iterating Proposition \ref{the recursion}, we obtain the following corollary.  

\begin{corollary}\label{iterated}
For any $\cF\in \Flag_k$, we have
$$\partial_{F_1}\cdots\partial_{F_{k-1}} \Snap_M^{\FY}(\textbf{a}) = \prod_{i=1}^k \Snap_{M^{F_{i}}_{F_{i-1}}}^{\FY}(\textbf{a}|_{[F_{i-1}, F_i]}).$$
\end{corollary}

\begin{proof}[Proof of Theorem \ref{thm:snapper-FY-wonderful}.]
We use the fact that
$$\partial_{F_1}\cdots\partial_{F_{k-1}} 
\left(a_{F_0}^{(m_0)}a_{F_1}^{(m_1+1)}a_{F_2}^{(m_2+1)}\cdots a_{F_{k-1}}^{(m_{k-1}+1)}a_{F_k}^{(m_k)}\right)
= a_{F_0}^{(m_0)}\cdots a_{F_k}^{(m_k)}$$
to interpret $c(\cF,\textbf{m})$ as 
the coefficient of $a_{F_0}^{(m_0)}\cdots a_{F_k}^{(m_k)}$ in the polynomial 
$\partial_{F_1}\cdots\partial_{F_{k-1}}\Snap_M^{\FY}(\textbf{a})$.
Corollary \ref{iterated} provides us with a formula for  this polynomial.
Setting $a_F = 0$ for all $F$ not appearing in $\cF$, 
we obtain the polynomial $$\prod_{i=1}^k \Snap_{M^{F_{i}}_{F_{i-1}}}^{\FY}(a_{F_{i-1}}, \mathbf{0}, a_{F_i}).$$
By Lemma \ref{ups}, we may rewrite this as
\begin{equation}\label{yuck}\sum_{\textbf{e},\textbf{f}\in\mathbb{N}^k}(-1)^{|\textbf{e}|+|\textbf{f}|}\prod_{i=1}^k \binom{d_i-e_i}{f_i}\mu^{e_i}(M^{F_{i}}_{F_{i-1}}) a_{F_{i-1}}^{(e_i)} a_{F_i}^{(f_i)}.\end{equation}
To compute the coefficient of $a_{F_0}^{(m_0)}\cdots a_{F_k}^{(m_k)}$,
we make use of the identity $$x^{({m})}x^{(n)} = \sum_{\ell = 0}^m\binom{m+n-\ell}{\ell,m-\ell,n-\ell}x^{(m+n-\ell)}$$
to rewrite \eqref{yuck} as
$$\sum_{\textbf{e},\textbf{f}\in\mathbb{N}^k}(-1)^{|\textbf{e}|+|\textbf{f}|}a_{\varnothing}^{(e_1)}\prod_{i=1}^k \binom{d_i-e_i}{f_i}\mu^{e_i}(M^{F_{i}}_{F_{i-1}}) \sum_{\ell_i=0}^{e_{i+1}} \binom{e_{i+1}+f_i-\ell_i}{\ell_i,e_{i+1}-\ell_i,f_i-\ell_i}a_{F_i}^{(e_{i+1} + f_i - \ell_i)},$$
with the convention that $e_{k+1} = 0$.
The coefficient of $a_{F_0}^{(m_0)}\cdots a_{F_k}^{(m_k)}$ consists of those terms for which
$e_1 = m_0$ and $e_{i+1}+f_i-\ell_i = m_i$ for all $0\leq i\leq k$.  
Thus $$c(\cF, \textbf{m}) = \sum_{\substack{\textbf{e},\textbf{f}\in\mathbb{N}^k\\ e_1 = m_0}}(-1)^{|\textbf{e}|+|\textbf{f}|}\prod_{i=1}^k \binom{d_i-e_i}{f_i}\mu^{e_i}(M^{F_{i}}_{F_{i-1}}) \binom{m_i}{e_{i+1}+f_i-m_i,m_i-f_i,m_i-e_{i+1}}. \qedhere$$
\end{proof}

\section{Euler characteristics and Snapper polynomials for $\overline{\calM}_{0, n}$}\label{mzeron snap}
Last, we turn our attention to the moduli space $\overline{\calM}_{0, n}$.
Recall that we have defined a tuple of line bundles 
$$\cL \coloneqq (\cL_S\mid S\subset [n-1], |S|\ge 3).$$
We begin this section by computing the Snapper polynomial $\Snap_\cL(\textbf{a})$ with respect to these line bundles.
Given a tuple of natural numbers $\textbf{m}  = (m_S \mid S \subset [n-1] \mid |S| \ge 3)$, we say that 
$\textbf{m}$ satisfies the {\bf Cerberus condition} if, for every $\textbf{m}'\leq \textbf{m}$,
we have $$\Big{|}\bigcup_{m'_S>0} S\cup\{n\}\,\,\Big{|} - 3\; \geq \;\sum_S m'_S.$$

\begin{theorem}\label{thm:mzeron-snapper}
We have $$\Snap_{\cL}(\textbf{a})\;\; = \sum_{\substack{\text{$\textbf{m}$ satisfies}\\\text{Cerberus}}} \textbf{a}^{(\textbf{m})}.$$
\end{theorem}

\begin{proof}
Let $\textbf{m}$ be given, and define
$$\tilde{\textbf{m}} = (\tilde{m}_F \mid \text{$F$ a nonempty flat of $\mathcal{B}_{n-1}$})$$
by putting $\tilde m_{F_S} = m_S$ and $\tilde m_F=0$
for all flats $F$ not of the form $F_S$.
Because the map $p\colon W_{\calB_{n-1}}\to \overline{\calM}_{0, n}$ is a composition of blowups at smooth centers,
the Euler characteristic of a line bundle on $\overline{\calM}_{0, n}$ is equal to the Euler characteristic
of its pullback to $W_{\calB_{n-1}}$. Thus, if we can show
that the Cerberus condition for $\textbf{m}$ is equivalent to the dragon Hall--Rado condition for $\tilde{\textbf{m}}$, 
our result will follow from Theorem \ref{thm:euler-simplicial-matroid}. Note that $\tilde{\textbf{m}}$ fails the dragon Hall--Rado condition if and only if there is $\tilde{\textbf{m}}' \le \tilde{\textbf{m}}$ with
$$\operatorname{rank}\bigcup_{\tilde m'(F)>0} F \;\le\; \sum_{F \not= \emptyset} \tilde m'(F) \text{ and } \bigcup_{\tilde m'(F)>0} F \text{ is a connected set}.$$ 
The remainder of the proof follows by direct
calculation, using the facts that $\rk F_S = |S|-1$
and $\rk F_S \cup F_T = \rk F_{S\cup T}$ if $S$ and $T$ are not disjoint. 
\end{proof}

We conclude by computing the Snapper polynomial with respect to a different tuple of line bundles 
on $\overline{\calM}_{0, n}$.
Recall that $\mathbb{L}_i$ is the $i^\text{th}$ cotangent line bundle, whose first Chern class is equal to $\psi_i$, and let
$\mathbb{L} \coloneqq (\mathbb{L}_1,\ldots,\mathbb{L}_n).$ In \cite{Pandharipande}, Pandharipande showed that if $a_1, \dotsc, a_{n} \ge 0$, then $H^i(\overline{M}_{0,n}, \mathbb{L}_1^{\otimes a_1} \otimes \dotsb \otimes \mathbb{L}_n^{\otimes a_n}) = 0$ for $i > 0$, and so 
$$\chi(\overline{M}_{0,n}, \mathbb{L}_1^{\otimes a_1} \otimes \dotsb \otimes \mathbb{L}_n^{\otimes a_n}) = h^0(\overline{M}_{0,n}, \mathbb{L}_1^{\otimes a_1} \otimes \dotsb \otimes \mathbb{L}_n^{\otimes a_n}).$$
In \cite{Lee}, Lee gave an expression for the generating function of this Euler characteristic, which is equivalent to the following theorem. 

\begin{theorem}\cite{Lee}\label{psi snap}
\label{thm:psi-classes}
We have
$$\Snap_\mathbb{L}(\textbf{a}) = \sum_{|\textbf{d}|\leq n-3} \binom{a_1}{d_1}\cdots\binom{a_n}{d_n}\binom{n-3}{n-3-\abs{\textbf{d}}, d_1, \ldots, d_n}.$$
\end{theorem}

We give a new proof of Theorem \ref{psi snap}, beginning with the following lemma.  Recall the exceptional isomorphism $\zeta_n\colon K(\overline{\calM}_{0, n}) \to A(\overline{\calM}_{0, n})$
of Theorem \ref{thm:mzeron}.  Let $z_i \coloneqq 1 - [\mathbb{L}_i]\in K(\overline{\calM}_{0, n})$.

\begin{lemma}\label{zeta psi}
We have $\zeta_n(z_i) = -\psi_i$ for all $i<n$, and $\zeta_n(z_n) = -\psi_n/(1-\psi_n)$. 
\end{lemma}

\begin{proof}
We will consider a family of closely related moduli spaces of curves, called the Losev--Manin spaces \cite{LosevManin}. 
The Losev--Manin space $\overline{\calM}_{0, w_{k, \ell}}$ is a moduli space of stable rational curves with $n$ weighted 
marked points \cite{Hassett}, with weights 1 for the the $k^\text{th}$ and $\ell^\text{th}$ points and 
$\varepsilon \in \Q \cap (0, \frac{1}{n-2})$ for the remaining $n-2$ points. 
Reduction of weights from $(1,\ldots,1)$ to $w_{k, \ell}$ induces a morphism \cite[Theorem 4.1]{Hassett}
$\rho_{k, \ell}\colon \overline{\calM}_{0, n}\to\overline{\calM}_{0, w_{k, \ell}}.$
Furthermore, $\overline{\calM}_{0, w_{k, \ell}}$ is isomorphic to the permutohedral variety $X_{{[n] \smallsetminus \{k, \ell\}}}$
\cite{LosevManin}, so we have an exceptional isomorphism
$$\zeta_{[n] \smallsetminus \{k, \ell\}} \colon K(\overline{\calM}_{0, w_{k, \ell}}) \to A(\overline{\calM}_{0, w_{k, \ell}}).$$ 
For $i\not= n$, consider the classes  $t_E, t_{\emptyset} \in A(\overline{\calM}_{0, w_{i, n}})\cong A(X_{{[n] \smallsetminus \{i, n\}}})$  and $\tau_E, \tau_{\emptyset} \in K(\overline{\calM}_{0, w_{i, n}})\cong K(X_{{[n] \smallsetminus \{i, n\}}})$.
Then we have the following identities:
$$\begin{aligned}
\rho_{n, i}^{\ast}t_E &= -\psi_n  & \quad\quad
\rho_{n, i}^{\ast}\tau_E &= 1 - [\mathbb{L}_n]\\
\rho_{n, i}^{\ast}t_{\emptyset} &= -\psi_i & \quad \quad
\rho_{n, i}^{\ast}\tau_{\emptyset} &= 1 - [\mathbb{L}_i],
\end{aligned}$$
see, for example, \cite[Section 2]{Ross}.
For each $i \ne n$, we have the following diagram
\[
\begin{tikzcd}
    K(\overline{\calM}_{0, w_{n, i}}) \arrow{r}{\zeta_{[n] \smallsetminus \{n, i\}}} \arrow{d}{\rho_{n, i}^{\ast}} & A(\overline{\calM}_{0, w_{n, i}}) \arrow{d}{\rho_{n, i}^{\ast}} \\
    K(\overline{\calM}_{0, n}) \arrow{r}{\zeta_n} & A(\overline{\calM}_{0, n}),
\end{tikzcd}
\] 
whose commutativity can be checked by using that $\rho_{n, i}^* \tau_E = 1 - [\mathbb{L}_n^{-1}]$, and that similar formulas hold for pullbacks of the $n$th cotangent line bundle under forgetful maps. 

The Lemma \ref{exceptional FY} gives that
$$\zeta_{n}(z_i) = \zeta_n(1-[\mathbb{L}_i])
= \zeta_n\circ\rho_{n, i}^{\ast}(\tau_{\emptyset})
= \rho_{n, i}^{\ast}\circ\zeta_{[n] \smallsetminus \{n, i\}}(\tau_{\emptyset})
= \rho_{n, i}^{\ast}(t_{\emptyset})
= -\psi_i,$$
and similarly 
$\zeta_n(z_n) = \rho_{n, i}^{\ast} (\frac{t_E}{1 + t_E}) = \frac{-\psi_n}{1-\psi_n}.$
\end{proof}

\begin{proof}[Proof of Theorem \ref{psi snap}.]
By Lemma \ref{oh snap}, we have
$$\Snap_\mathbb{L}(-\textbf{a}) = \sum_{\textbf{d}\in\mathbb{N}^n}\chi\left(\overline{\calM}_{0, n},\, z_1^{d_1}\cdots z_n^{d_n}\right)\, \textbf{a}^{(\textbf{d})}.$$
(Note that the minus sign comes from the fact that $z_i = 1 - [\mathbb{L}_i]$ rather than $1-[\mathbb{L}_i^{-1}]$.)
By Theorem \ref{thm:mzeron} and Lemma \ref{zeta psi}, we have 
$$\chi\left(\overline{\calM}_{0, n},\, z_1^{d_1}\cdots z_n^{d_n}\right)
= \deg \left (\frac{\zeta_n(z_1^{d_1}\cdots z_n^{d_n})}{1-c_{[n-1]}} \right)
= (-1)^{|\textbf{d}|} \deg\left(\frac{\psi_1^{d_1}\cdots\psi_{n}^{d_{n}}}{(1-\psi_n)^{d_n+1}}\right).$$
Next, we observe that
$$\frac{\psi_1^{d_1}\cdots\psi_{n}^{d_{n}}}{(1-\psi_n)^{d_n+1}} = \psi_1^{d_1}\cdots\psi_{n}^{d_{n}}\sum_{k=0}^\infty (d_n+1)^{(k)} \psi_n^k,$$
and therefore the part of this sum in degree $n-3$ is equal to
$$\psi_1^{d_1}\cdots\psi_{n-1}^{d_{n-1}}\psi_n^{n-3-d_1-\cdots-d_{n-1}}(d_n+1)^{(n-3-|\textbf{d}|)}$$ if $|\textbf{d}|\leq n-3$, and zero otherwise.
By work of Witten \cite{MR1144529}, the degree of this term is equal to
$$\binom{n-3}{d_1,\ldots,d_{n-1},n-3-d_1-\cdots-d_{n-1}}(d_n+1)^{(n-3-|\textbf{d}|)}.$$
Putting it all together, we have
\begin{eqnarray*}
\Snap_\mathbb{L}(-\textbf{a}) &=& \sum_{d\in\mathbb{N}^n}\chi\left(\overline{\calM}_{0, n},\, z_1^{d_1}\cdots z_n^{d_n}\right)\, \textbf{a}^{(\textbf{d})}\\
&=& \sum_{|\textbf{d}|\leq n-3} (-1)^{|\textbf{d}|} \textbf{a}^{(\textbf{d})}\deg\left(\frac{\psi_1^{d_1}\cdots\psi_{n}^{d_{n}}}{(1-\psi_n)^{d_n+1}}\right)\\
&=& \sum_{|\textbf{d}|\leq n-3} (-1)^{|\textbf{d}|} \textbf{a}^{(\textbf{d})}(d_n+1)^{(n-3-|\textbf{d}|)}\binom{n-3}{d_1,\ldots,d_{n-1},n-3-d_1-\cdots-d_{n-1}}\\
&=& \sum_{|\textbf{d}|\leq n-3} (-1)^{|\textbf{d}|} \textbf{a}^{(\textbf{d})}\binom{d_n+n-3-|\textbf{d}|}{n-3-|\textbf{d}|}\binom{n-3}{d_1,\ldots,d_{n-1},n-3-d_1-\cdots-d_{n-1}}\\
&=& \sum_{|\textbf{d}|\leq n-3} (-1)^{|\textbf{d}|} \textbf{a}^{(\textbf{d})}\binom{n-3}{d_1,\ldots,d_{n},n-3-|\textbf{d}|}.
\end{eqnarray*}
Finally, this implies that 
$$\Snap_\mathbb{L}(\textbf{a}) = \sum_{|\textbf{d}|\leq n-3} \binom{a_1}{d_1}\cdots\binom{a_n}{d_n}\binom{n-3}{d_1,\ldots,d_{n},n-3-|\textbf{d}|},$$
which completes the proof.
\end{proof}

\appendix
\section{The simplicial presentation of Chow rings}\label{sec:augChow}
This appendix is devoted to proving the simplicial presentation for the (augmented) Chow ring of a matroid. 
The simplicial generators of $A(M)$ were extensively studied in \cite{BES}, and the surjectivity of $\Phi_{\nabla}$ was proved there. This immediately generalizes to $\Phi_{\nabla}^{\aug}$. That work did not give a simple description of the kernel of $\Phi_{\nabla}$. 

For any pair of flats $F$ and $G$, let $$z_{F,G} \coloneqq \sum_{\substack{F\subset F'\subsetneq F\vee G\\ G\subset G'\subsetneq F\vee G}} x_{F'}x_{G'}\in T_M.$$
For any element $e\in E$ and flat $F\neq \emptyset$, let $$w_{e,F} \coloneqq \sum_{F\subset G \not\ni e} y_ex_{G} 
\in T_M.$$
Recall the definitions of the ideals $\cI_1,\cI_2,\cI_3,\cI_4^\aug\subset T_M$ in Equation \eqref{eq:Idef}.

\begin{lemma}\label{change of coords}
We have the following equalities of ideals in $T_M$:
\begin{eqnarray*}
\cI_3 &=& \left\langle z_{F,G}\mid \text{$F, G$ {\em arbitrary}} \right \rangle\\
\cI_2 + \cI_4^\aug&=& \cI_2 +\langle y_e^2\mid e\in E\rangle +  \left\langle w_{e,F}\mid e\in E,  F\neq \varnothing \right \rangle.
\end{eqnarray*}
\end{lemma}

\begin{proof}
We begin with the first statement.
Every term in $z_{F,G}$ is a product of incomparable flats, thus $z_{F,G}$ is contained in $\cI_3$.
The opposite inclusion follows by induction on sum of the coranks of $F$ and $G$, using the observation that, whenever $F$ and $G$ are incomparable,
$$z_{F,G} = x_F x_G + \text{terms with lower sum of coranks.}$$

For the second statement, we write $$\cI_4^\aug = \langle y_e x_\varnothing \mid e\in E\rangle + \langle y_e x_F \mid e\notin F\neq\varnothing\rangle.$$
A similar argument to the one above shows that 
$$\langle y_e x_F \mid e\notin F\neq\varnothing\rangle = \left\langle w_{e,F}\mid e\in E,  F\neq \varnothing \right \rangle,$$
so it remains only to show that $y_e x_\varnothing$ is congruent to $y_e^2$ modulo the ideal $\cI_2 +\left\langle w_{e,F}\mid e\in E,  F\neq \varnothing \right \rangle$.
Indeed, we have $$y_e^2 - x_\varnothing y_e = y_e\left( y_e - \sum_{e\notin F} x_F\right) + \sum_{e\notin F\neq \varnothing} y_e x_F,$$
which completes the proof.
\end{proof}

Recall that we have defined the homomorphism $\Phsim^\aug\colon S_M\to A^\aug(M)$ by the formula
$$\Phsim^\aug(u_F) \coloneqq -\sum_{F\subset G}\PhFY(x_G).$$
It is easy to see that the kernel of the map $A^{\aug}(M) \to A(M)$ is generated by $\langle h_F : \operatorname{rk} F = 1\rangle$, so the following theorem implies Theorem~\ref{Amatroid}. 

\begin{theorem}\label{thm:augmentedsimplicial}
The map $\Phsim^\aug$ is surjective with kernel $\mathcal{J}_1 + \mathcal{J}_2^{\aug}$.
\end{theorem}

\begin{proof}
Consider the homomorphism $\Theta\colon S_M\to T_M$ defined by the formula $$\Theta(u_F) \coloneqq -\sum_{F\subset G} x_G$$
for all nonempty flats $F$, so that $\Phsim^\aug = \PhFY^\aug\circ\Theta$.
The image of $\Theta$ is equal to the subring $\Z[x_F\mid \text{$F$ a nonempty flat}] \subset T_M$.  In particular, the composition $$\bar\Theta\colon S_M\to T_M \to T_M/(\cI_1+\cI_2)$$ is
an isomorphism. This implies that $\Phsim^\aug$ is surjective, 
and its kernel is equal to $$\Theta^{-1}(\cI_1 + \cI_2 + \cI_3 + \cI_4^\aug).$$

For the remainder of the proof, we put a bar over an element of $T_M$ or an ideal in $T_M$ to denote its image in $T_M/(\cI_1+\cI_2)$.
Then we need to compute $\bar{\Theta}^{-1}(\bar\cI_3 + \bar\cI_4^\aug)$.
By Lemma \ref{change of coords}, we may rewrite this as 
$$\bar{\Theta}^{-1}\Big(\left\langle \bar z_{F,G}\mid \text{$F, G$ arbitrary} \right \rangle 
+ \langle \bar y_e^2\mid e\in E\rangle +  \left\langle w_{e,F}\mid e\in E,  F\neq \varnothing \right \rangle\Big).$$
We have the following equalities:
\begin{eqnarray*} \bar z_{F,G} &=&  \bar\Theta\Big((u_F - u_{F\vee G})(u_G - u_{F\vee G})\Big),\\
\bar y_e^2 &=& \bar\Theta(u_{\bar e}^2),\\
\bar w_{e,F} &=& \bar\Theta\Big(u_{\bar{e}}(u_{\bar{e}\vee F}-u_F)\Big),
\end{eqnarray*}
which together imply that the kernel of $\Phsim^\aug$ is equal to
$$\cJ_1 + \langle u_{\bar e}^2\mid e\in E\rangle + \langle u_{\bar{e}}(u_{\bar{e}\vee F}-u_F) \mid e\in E, F\neq \varnothing\rangle.$$
The theorem now follows from the fact that the second and third summand above add to $\cJ_2^\aug$.
\end{proof}

\bibliographystyle{alpha}
\bibliography{wondk.bib}

\newcommand{\etalchar}[1]{$^{#1}$}
\begin{thebibliography}{BHM{\etalchar{+}}22b}

\bibitem[AHK18]{AHK}
Karim Adiprasito, June Huh, and Eric Katz.
\newblock Hodge theory for combinatorial geometries.
\newblock {\em Ann. Math.}, 188(2):381--452, 2018.

\bibitem[AK06]{ardila2006bergman}
Federico Ardila and Caroline~J. Klivans.
\newblock The {B}ergman complex of a matroid and phylogenetic trees.
\newblock {\em J. Combin. Theory Ser. B}, 96(1):38--49, 2006.

\bibitem[AP15]{anderson2015operational}
Dave Anderson and Sam Payne.
\newblock Operational {$K$}-theory.
\newblock {\em Doc. Math.}, 20:357--399, 2015.

\bibitem[AS23]{AS22}
Federico Ardila and Mario Sanchez.
\newblock Valuations and the {H}opf monoid of generalized permutahedra.
\newblock {\em Int. Math. Res. Not. IMRN}, (5):4149--4224, 2023.

\bibitem[BES23]{BES}
Spencer Backman, Christopher Eur, and Connor Simpson.
\newblock Simplicial generation of chow rings of matroids.
\newblock {\em J. Eur. Math.}, 2023.

\bibitem[BEST23]{BEST}
Andrew Berget, Christopher Eur, Hunter Spink, and Dennis Tseng.
\newblock Tautological classes of matroids.
\newblock {\em Invent. Math.}, 233(2):951--1039, 2023.

\bibitem[BGI71]{SGA6}
Pierre Berthelot, Alexandre Grothendieck, and Luc Illusie.
\newblock {\em Th\'eorie des intersections et th\'eor\`eme de {R}iemann-{R}och}.
\newblock Lecture Notes in Mathematics, Vol. 225. Springer-Verlag, Berlin, 1971.
\newblock S{\'e}minaire de G{\'e}om{\'e}trie Alg{\'e}brique du Bois-Marie 1966--1967 (SGA 6).

\bibitem[BHM{\etalchar{+}}22a]{BHMPW20a}
Tom Braden, June Huh, Jacob Matherne, Nicholas Proudfoot, and Botong Wang.
\newblock A semi-small decomposition of the {C}how ring of a matroid.
\newblock {\em Adv. Math.}, 409:Paper No. 108646, 2022.

\bibitem[BHM{\etalchar{+}}22b]{BHMPW20b}
Tom Braden, June Huh, Jacob Matherne, Nicholas Proudfoot, and Botong Wang.
\newblock {S}ingular {H}odge theory for combinatorial geometries.
\newblock Preprint, arXiv:2010.06088v3, 2022.

\bibitem[CHMR16]{CHMR}
Renzo Cavalieri, Simon Hampe, Hannah Markwig, and Dhruv Ranganathan.
\newblock Moduli spaces of rational weighted stable curves and tropical geometry.
\newblock {\em Forum Math. Sigma}, 4:Paper No. e9, 35, 2016.

\bibitem[CT20]{castravet2020derived}
Ana-Maria Castravet and Jenia Tevelev.
\newblock Derived category of moduli of pointed curves. {I}.
\newblock {\em Algebraic Geometry}, 7(6):722--757, 2020.

\bibitem[CT21]{CT2}
Ana-Maria Castravet and Jenia Tevelev.
\newblock Derived category of moduli of pointed curves. {II}.
\newblock Preprint, arXiv:2002.02889v3, 2021.

\bibitem[DCP95]{dCP95}
Corrado De~Concini and Claudio Procesi.
\newblock Wonderful models of subspace arrangements.
\newblock {\em Selecta Math. (N.S.)}, 1(3):459--494, 1995.

\bibitem[DF10]{DF10}
Harm Derksen and Alex Fink.
\newblock Valuative invariants for polymatroids.
\newblock {\em Adv. Math.}, 225(4):1840--1892, 2010.

\bibitem[DR22]{Ross}
Jeshu Dastidar and Dustin Ross.
\newblock Matroid psi classes.
\newblock {\em Selecta Math. (N.S.)}, 28(3):Paper No. 55, 38, 2022.

\bibitem[EHKR10]{MR2630055}
Pavel Etingof, Andr\'{e} Henriques, Joel Kamnitzer, and Eric Rains.
\newblock The cohomology ring of the real locus of the moduli space of stable curves of genus 0 with marked points.
\newblock {\em Ann. Math.}, 171(2):731--777, 2010.

\bibitem[EHL23]{EurHuhLarson}
Christopher Eur, June Huh, and Matt Larson.
\newblock Stellahedral geometry of matroids.
\newblock {\em Forum Math. Pi}, 11:Paper No. e24, 48, 2023.

\bibitem[EL23]{EHLPolymatroid}
Christopher Eur and Matt Larson.
\newblock Intersection theory of polymatroids.
\newblock 2023.
\newblock Preprint, ar{X}iv:2301:00831v2, Int. Math. Res. Not. IMRN (to appear).

\bibitem[Eur20]{Eur20}
Christopher Eur.
\newblock Divisors on matroids and their volumes.
\newblock {\em J. Combin. Theory Ser. A}, 169:105135, 31, 2020.

\bibitem[FM05]{FeichtnerMuller}
Eva~Maria Feichtner and Irene M\"{u}ller.
\newblock On the topology of nested set complexes.
\newblock {\em Proc. Amer. Math. Soc.}, 133(4):999--1006, 2005.

\bibitem[Ful93]{Ful93}
William Fulton.
\newblock {\em {Introduction to toric varieties}}, volume 131 of {\em Annals of Mathematics Studies}.
\newblock Princeton University Press, Princeton, NJ, 1993.

\bibitem[Ful98]{Fulton}
William Fulton.
\newblock {\em Intersection theory}, volume~2 of {\em Ergebnisse der Mathematik und ihrer Grenzgebiete. 3. Folge. A Series of Modern Surveys in Mathematics}.
\newblock Springer-Verlag, Berlin, second edition, 1998.

\bibitem[FY04]{FY}
Eva~Maria Feichtner and Sergey Yuzvinsky.
\newblock Chow rings of toric varieties defined by atomic lattices.
\newblock {\em Invent. Math.}, 155(3):515--536, 2004.

\bibitem[Get95]{Getzler}
Ezra Getzler.
\newblock Operads and moduli spaces of genus {$0$} {R}iemann surfaces.
\newblock In {\em The moduli space of curves ({T}exel {I}sland, 1994)}, volume 129 of {\em Progr. Math.}, pages 199--230. Birkh\"{a}user, Boston, MA, 1995.

\bibitem[Has03]{Hassett}
Brendan Hassett.
\newblock Moduli spaces of weighted pointed stable curves.
\newblock {\em Adv. Math.}, 173(2):316--352, 2003.

\bibitem[Kap93]{Kapranov}
Mikhail Kapranov.
\newblock Chow quotients of {G}rassmannians. {I}.
\newblock In {\em I. {M}. {G}elfand {S}eminar}, volume~16 of {\em Adv. Soviet Math.}, pages 29--110. Amer. Math. Soc., Providence, RI, 1993.

\bibitem[KKL21]{Kannan2021ChowRO}
Siddarth Kannan, Dagan Karp, and Shiyue Li.
\newblock Chow rings of heavy/light {H}assett spaces via tropical geometry.
\newblock {\em J. Combin. Theory Ser. A}, 178(1):Paper No. 105348, 2021.

\bibitem[Lee97]{Lee}
Yuan-Pin Lee.
\newblock A formula for {E}uler characteristics of tautological line bundles on the {D}eligne-{M}umford moduli spaces.
\newblock {\em Int. Math. Res. Not. IMRN}, (8):393--400, 1997.

\bibitem[LM00]{LosevManin}
Andrey Losev and Yuri Manin.
\newblock New moduli spaces of pointed curves and pencils of flat connections.
\newblock {\em Michigan Math. J.}, 48:443--472, 2000.

\bibitem[McM09]{McMullen2009}
Peter McMullen.
\newblock Valuations on lattice polytopes.
\newblock {\em Adv. Math.}, 220(1):303--323, 2009.

\bibitem[Pan97]{Pandharipande}
Rahul Pandharipande.
\newblock The symmetric function {$h^0(\overline M_{0,n},\mathscr L^{x_1}_1\otimes\mathscr L^{x_2}_2\otimes\cdots\otimes\mathscr L^{x_n}_n)$}.
\newblock {\em J. Algebraic Geom.}, 6(4):721--731, 1997.

\bibitem[Pos09]{Postnikov}
Alexander Postnikov.
\newblock Permutohedra, associahedra, and beyond.
\newblock {\em Int. Math. Res. Not. IMRN}, 2009(6):1026--1106, 2009.

\bibitem[Sin04]{singh}
Daniel Singh.
\newblock {\em The moduli space of stable {$N$}-pointed curves of genus zero}.
\newblock PhD thesis, University of Sheffield, 2004.

\bibitem[Sna59]{MR0109156}
Ernst Snapper.
\newblock Multiples of divisors.
\newblock {\em J. Math. Mech.}, 8:967--992, 1959.

\bibitem[Wei94]{MR1269324}
Charles Weibel.
\newblock {\em An introduction to homological algebra}, volume~38 of {\em Cambridge Studies in Advanced Mathematics}.
\newblock Cambridge University Press, Cambridge, 1994.

\bibitem[Wit91]{MR1144529}
Edward Witten.
\newblock Two-dimensional gravity and intersection theory on moduli space.
\newblock In {\em Surveys in differential geometry ({C}ambridge, {MA}, 1990)}, pages 243--310. Lehigh Univ., Bethlehem, PA, 1991.

\bibitem[Yuz02]{YuzvinskySimplicial}
Sergey Yuzvinsky.
\newblock Small rational model of subspace complement.
\newblock {\em Trans. Amer. Math. Soc.}, 354(5):1921--1945, 2002.

\end{thebibliography}
\end{document}